\documentclass[11pt]{article}

\usepackage{graphicx}
\usepackage{wrapfig}
\usepackage{substr}
\usepackage[tight]{minitoc}
\usepackage{subfigure}
\usepackage{longtable}
\usepackage[all]{xy}
\usepackage{xcite}
\usepackage{xr}

\usepackage{amsfonts}
\usepackage{nicefrac}

\usepackage{amsmath,amssymb,amsthm}
\usepackage{mathtools}

\newcommand{\sR}{{\R^{d}}}
\newcommand{\SM}{\mathbb{S}} 
\newcommand{\N}{\mathbb{N}}
\newcommand{\sX}{{\R^{d_x}}}
\newcommand{\sY}{{\R^{d_y}}}
\newcommand{\conv}[1]{\operatorname{conv}(#1)}
\newcommand{\und}[2]{\underset{#2}{#1}}

\providecommand{\keywords}[1]{\text{\text{Keywords: }} #1}

\usepackage{fullpage}
\usepackage{cancel}
\usepackage[utf8]{inputenc} 
\usepackage[T1]{fontenc}    
\usepackage{hyperref}       
\usepackage{url}            
\usepackage{booktabs}       
\usepackage{amsfonts}       
\usepackage{nicefrac}       
\usepackage{microtype}      

\usepackage{authblk}
\usepackage{natbib}
\usepackage{sidecap}

\usepackage{amssymb, amsmath, amsthm, latexsym}
\usepackage{url}
\usepackage{algorithm}
\usepackage{algorithmic}
\usepackage{tabularx}
\usepackage{paralist}
\usepackage{mathtools}

\usepackage{bbm} 

\usepackage{makecell}
\usepackage{multirow}
\usepackage{booktabs}

\usepackage{nicefrac}       

\usepackage[flushleft]{threeparttable} 

\usepackage{caption}
\usepackage{subcaption}
\usepackage{multirow}
\usepackage{colortbl}
\definecolor{bgcolor}{rgb}{0.8,1,1}
\definecolor{bgcolor2}{rgb}{0.8,1,0.8}
\definecolor{niceblue}{rgb}{0.0,0.19,0.56}

\hypersetup{colorlinks,linkcolor={blue},citecolor={blue},urlcolor={blue}}

\usepackage[dvipsnames]{xcolor}
\usepackage{pifont}

\newcommand{\R}{\mathbb{R}}

\def\<#1,#2>{\left\langle #1,#2\right\rangle}

\usepackage{mdframed} 
\usepackage{thmtools}

\definecolor{shadecolor}{gray}{0.9}
\declaretheoremstyle[
headfont=\normalfont\bfseries,
notefont=\mdseries, notebraces={(}{)},
bodyfont=\normalfont,
postheadspace=0.5em,
spaceabove=1pt,
mdframed={
 skipabove=8pt,
 skipbelow=8pt,
 hidealllines=true,
 backgroundcolor={shadecolor},
 innerleftmargin=4pt,
 innerrightmargin=4pt}
]{shaded}

\declaretheorem[style=shaded,within=section]{definition}
\declaretheorem[style=shaded,sibling=definition]{theorem}

\declaretheorem[style=shaded,sibling=definition]{lemma}

\usepackage{xspace}

\usepackage[colorinlistoftodos,bordercolor=orange,backgroundcolor=orange!20,linecolor=orange,textsize=scriptsize]{todonotes}

\usepackage{hyperref}
\graphicspath{{plots/}}

\usepackage{makecell}

\usepackage{accents}
\newlength{\dhatheight}

\renewcommand{\geq}{\geqslant}

\newcommand\swapifbranches[3]{#1{#3}{#2}}
\makeatletter
\MHInternalSyntaxOn
\patchcmd{\DeclarePairedDelimiter}{\@ifstar}{\swapifbranches\@ifstar}{}{}
\MHInternalSyntaxOff
\makeatother

\DeclarePairedDelimiterX{\inp}[2]{\langle}{\rangle}{#1, #2}
\DeclarePairedDelimiterX{\abs}[1]{\lvert}{\rvert}{#1}
\DeclarePairedDelimiterX{\roundup}[1]{\lceil}{\rceil}{#1}
\DeclarePairedDelimiterX{\cbr}[1]{\{}{\}}{#1} 
\DeclarePairedDelimiterX{\rbr}[1]{(}{)}{#1} 
\DeclarePairedDelimiterX{\sbr}[1]{[}{]}{#1} %

\title{\bf Convex-Concave Interpolation and Application of PEP to Bilinear-Coupled Saddle-Point Problem}

\author[1]{Valery Krivchenko}
\author[2]{Alexander Gasnikov}
\author[3]{Dmitry Kovalev}

\affil[1]{Moscow Institute of Physics ans Technology,
Institutskiy Pereulok, 9, Dolgoprudny, Moscow Oblast, 141701, Russia}
\affil[2]{Innopolis University,
Universitetskaya St, 1, Innopolis, Republic of Tatarstan, 420500, Russia}
\affil[3]{Yandex Research, Ulitsa L'va Tolstogo, 16, Moscow, 119021, Russia}

\date{\today}

\begin{document}

\maketitle

\abstract{
The Performance estimation problem (PEP) approach reformulates finding the exact worst-case performance of an algorithm as the solution to an optimization problem. Tractable formulation of the problem requires necessary and sufficient interpolation conditions. We present the interpolation conditions for convex--concave functions, bilinear functions, and composite functions with bilinear coupling. We also construct PEP for first-order methods for composite saddle-point problem.}


\footnote{A version of this work containing preliminary numerical experiments appeared in Russian Journal of Nonlinear Dynamics. This extended version contains more detailed description of PEP construction and the program for automatic generation of Lyapunov functions.}

\section{Introduction}

Consider saddle-point problem (SPP):
\begin{align} \label{problem:minmax}
	\und{\min}{x \in \sX} \und{\max}{y \in \sY} f(x,y),
\end{align}
where $f \colon \sX \times \sY \to \R$ is convex in $\mathbf{x}$ and concave in $\mathbf{y}$. The problem is widely studied in optimization and arises, besides mathematics, in economics, computer science and machine learning~\cite{goodfellow2020generative,li2019robust} where SPPs tend to have large dimensions. In the latter, first-order gradient methods become dominant algorithmic paradigm due to low cost per iteration and scalability. Unfortunately, their exact worst-case performance remains largely unexplored \cite{lee2024fundamental,zhang2022near,zhang2021unified} in minimax setting.

There was developed a principled way to approach worst-case computation problems, also known as Performance estimation problems (PEP) \cite{drori2014performance}. Its idea is to express worst-case performance as a solution of a certain optimization problem. General formulation of PEP is infinite dimensional because the performance metric is being optimized over an entire class of functions. To ensure tractability, optimization over functions is replaced with optimization over iterates, gradient and function values restricted by the interpolation conditions \cite{taylor2017smooth}. The conditions must be necessary and sufficient for interpolation so that there is equivalence of the problems' optimal values. The approach yielded numerous results in convex minimization and helped to develop optimal methods \cite{kim2016optimized}. Meanwhile, its application for SPPs is held back by the lack of the relevant interpolation conditions.

In this paper, we present necessary and sufficient interpolation conditions for nonsmooth convex--concave functions, bilinear functions and composite functions with bilinear coupling $p(x) + y^\top Ax - q(y)$. The latter are of significant practical importance and arise in different applications \cite{pmlr-v37-zhanga15, pmlr-v70-scaman17a, peyre2019computational, chambolle2016introduction}. For composite SPPs, we construct standard and Lyapunov \cite{taylor2018lyapunov} formulations of PEP.

\section{Preliminaries}

We work in a finite-dimensional real vector space $\mathbb{R}^d$ equipped with the standard inner product $\langle \cdot, \cdot \rangle$ and the induced Euclidean norm $\|\cdot\|_2$.

We refer to standard literature for the comprehensive treatment of convex analysis~\cite{bauschke2011convex} and minimax theory~\cite{rockafellar1970}.

\begin{definition} \label{fed:smooth_sc} Function $f: \sR \rightarrow \R \cup \{+\infty\}$ is $L$-smooth ($L > 0$), $\mu$-strongly-convex ($\mu \geq 0$), or $f \in \mathcal F_{\mu, L}$, if
\begin{align*}
    f(x) - \frac{\mu\|x\|_2^2}{2} \quad \text{is convex}
\end{align*}
and $\forall x_0,x_1 \in \R$, and $\forall g_0 \in \partial f(x_0), \ g_1 \in \partial f(x_1)$:
\begin{align*}
    \|g_1-g_0 \|_2 \leqslant L\|x_1-x_0\|_2.
\end{align*}    
\end{definition}

\begin{definition} \label{fed:base_convex_concave} Function $f \colon \sX \times \sY \to \R$ is convex--concave ($f \in \mathcal S$) if 
\begin{align*}
&f(\cdot,y): \text{convex } \  \forall y \in \sY, \\
&f(x,\cdot): \text{concave } \ \forall x \in \sX.
\end{align*}
\end{definition}

\begin{definition} \label{fed:gen_SCSC_smooth}
$f \colon \sX \times \sY \to \R$ is $\left(\mu_x,\mu_y\right)$-strongly-convex-strongly-concave (SCSC) and has $\left(L_x,L_y,L_{xy}\right)$-Lipschitz gradients ($f \in \mathcal S_{\mu_x,\mu_y,L_x,L_y,L_{xy}}$) if
\begin{align*}
&f\left(\cdot,y\right): \mu_x-\text{strongly convex}\quad \forall y \in \sY,\\
&f\left(x,\cdot\right): \mu_y-\text{strongly concave}\quad \forall x \in \sX,\\
&\|\nabla_x f(x_1,y) - \nabla_x f(x_0,y)\|_2 \leqslant L_x \|x_1 - x_0\|_2  \quad \forall x_0,x_1 \in \sX, y \in \sY,\\
&\|\nabla_y f(x,y_1) - \nabla_y f(x,y_0)\|_2 \leqslant L_y \|y_1 - y_0\|_2  \quad \forall y_0,y_1 \in \sY, x \in \sX,\\
&\|\nabla_x f(x,y_1) - \nabla_x f(x,y_0)\|_2 \leqslant L_{xy} \|y_1 - y_0\|_2  \quad \forall y_0,y_1 \in \sY, x \in \sX,\\
&\|\nabla_y f(x_1,y) - \nabla_y f(x_0,y)\|_2 \leqslant L_{xy} \|x_1 - x_0\|_2  \quad \forall x_0,x_1 \in \sX, y \in \sY,
\end{align*}
where $0 \leqslant \mu_x \leqslant L_x$, $0 \leqslant \mu_y \leqslant L_y$, and $L_{xy} \geqslant 0$.
\end{definition}

We define interpolability similarly to \cite{taylor2017smooth}, where it was introduced for convex functions.

\begin{definition} \label{fed:interpolation}
Let $I$ be an index set and consider a sequence $A_I = \left\{(x_i, y_i, g^x_i, g^y_i, f_i)\right\}_{i \in I}$ where $x_i, g^x_i \in \sX$, $y_i, g^y_i \in \sY$, $f_i \in \R$ for all $i \in I$. Consider a set of convex--concave functions $\mathcal S$. Sequence $A_I$ is $\mathcal S$-interpolable if and only if there exists function $f \in S$ such that $g^x_i \in \partial_x f(x_i, y_i)$, $g^y_i \in \partial_y f(x_i, y_i)$, $f_i = f(x_i, y_i)$ for all $i \in I$.
\end{definition}

\section{Interpolation conditions}

\subsection{Convex--concave functions}

\begin{definition} \label{fed:lip} $f \in \mathcal S$ is $M$-Lipschitz (or, $f\in \mathcal S^{Lip}_M$) if
    \begin{equation*}
        |f(x_1,y_1) - f(x_0,y_0)| \leqslant M \left(\|x_1 - x_0\|_2 + \|y_1 - y_0\|_2\right)
    \end{equation*}
    for all $x_1,x_0 \in \sX$, $y_1,y_0 \in \sY$.
\end{definition}

The assumption of bounded subgradients enables the analysis of subgradient-type methods in the non-differentiable setting \cite{KallioRosa1999,Nedic2009Subgradient}. We allow $M = +\infty$ so that there is no loss of generality.

\begin{theorem} \label{teo:Lipschitz}
Sequence $A_I$ is $\mathcal S^{Lip}_M$-interpolable if and only if the following conditions are satisfied:
\begin{align}
    &f_i \geqslant f_j + \langle g^x_j, x_i-x_j\rangle + \langle g^y_i,y_i-y_j \rangle \quad \forall i,j \in I, 
        \label{eq5} \\
    & \|g^x_i\|_2 \leqslant M,\quad \|g^y_i\|_2\leqslant M \quad \forall i\in I. \label{eq6}
\end{align}
\end{theorem}

\begin{proof}
Necessity. Suppose $f \in \mathcal S^{Lip}_M$. From the definition of the subdifferentials $\partial_x f(x,y)$ and $\partial_y f(x,y)$ it follows that
\begin{align*}
	f_i + \langle g^y_i,y_j - y_i\rangle \geqslant f(x_i,y_j) \geqslant f_j + \langle g^x_j,x_i - x_j \rangle,
\end{align*}
which is inequality~\eqref{eq5}. Inequalities~\eqref{eq6} are a trivial consequence of $f$ being $M$-Lipschitz. \par

Sufficiency. Suppose that inequalities \eqref{eq5} and \eqref{eq6} are satisfied. Consider a function $f(x,y)$ which is defined as follows:
	\begin{equation}
		f(x,y) = \sup_{z \in \mathcal Z} \und{\min}{i \in I} \hspace{2pt} f_i + \langle z,x-x_i\rangle + \langle g^y_i,y-y_i\rangle,
	\end{equation}
	where the convex set $\mathcal Z \subset \sX$ is defined as follows
	\begin{equation}
		\mathcal Z = \conv{\bigcup_{i=1}^n \{g^x_i\}}.
	\end{equation}

	First, it is easy to observe that function $f(x,y)$ is convex in $x$.
	Indeed for fixed $y$, function $f(x,y)$ is a pointwise supremum of functions that are linear in $x$:
	\begin{equation*}
		f(x,y) = \sup_{z \in \mathcal Z} \left[\langle z, x \rangle + \varphi_y(z)\right],
	\end{equation*}
	where function $\varphi_y(z)$ is defined as follows:
	\begin{equation}
		\varphi_y(z) = \und{\min}{i \in I}\hspace{2pt} f_i - \langle z,x_i \rangle + \langle g^y_i, y-y_i\rangle.
	\end{equation}

	Next, one can show that the function $-f(x,y)$ is convex in $y$.
	Indeed,
	\begin{equation}
		-f(x,y) = \inf_{z \in \mathcal Z} \varphi_{x}(z,y),
	\end{equation}
	where function $\varphi_{x}(z,y)$ is defined as follows:
	\begin{equation}
		\varphi_{x}(z,y) = \und{\max}{i \in I} -\left[f_i + \langle z,x-x_i\rangle + \langle g^y_i, y-y_i \rangle\right].
	\end{equation}
	Note, that function $\varphi_{x}(z,y)$ is convex in $(z,y)$, which implies the convexity of the function $-f(x,y)$ in $y$.

	Now, we show that $f(x_k,y_k) = f_k$. On one hand, we get
	\begin{align*}
		f(x_k,y_k) & = \sup_{z \in \mathcal Z} \und{\min}{i \in I} \hspace{2pt} f_i + \langle z, x_k-x_i \rangle + \langle g^y_i,y_k-y_i\rangle
		\\&\leqslant
		\sup_{z \in \mathcal Z} f_k + \langle z, x_k-x_k\rangle + \langle g^y_k, y_k-y_k \rangle 
		\\&= f_k.
	\end{align*}
	On the other hand, we get
	\begin{align*}
		f(x_k,y_k) & = \sup_{z \in \mathcal Z} \und{\min}{i \in I} f_i + \langle z, x_k-x_i \rangle + \langle g^y_i, y_k-y_i \rangle
		\\&\geqslant
		\und{\min}{i \in I} f_i + \langle g^x_k, x_k-x_i \rangle + \langle g^y_i, y_k-y_i \rangle
		\\&\geqslant \und{\min}{i \in I} f_k
		\\&= f_k,
	\end{align*}
	where we used \eqref{eq5} in the last inequality.
	Hence, $f(x_k,y_k) = f_k$.

	Next, we show that $g^x_k \in \partial_x f(x_k,y_k)$. Indeed,
\begin{align*}
		f(x,y_k) - f(x_k,y_k) - \langle g^x_k, x-x_k\rangle 
		 & =
		\sup_{z \in \mathcal Z} \und{\min}{i \in I} f_i + \langle z, x-x_i \rangle + \langle g^y_i, y_k-y_i \rangle \\&\quad - f_k - \langle g^x_k, x-x_k \rangle
		\\&\geqslant
		\und{\min}{i \in I} f_i + \langle g^x_k, x-x_i \rangle + \langle g^y_i, y_k-y_i \rangle \\&\quad - f_k - \langle g^x_k, x-x_k \rangle
		\\&=
		\und{\min}{i \in I} f_i  + \langle g^y_i, y_k-y_i\rangle - \left(f_k + \langle g^x_k, x_i-x_k \rangle \right)
		\\&\geqslant 0.
\end{align*}
Hence, $g^x_k \in \partial_x f(x_k,y_k)$ by the definition of the subdifferential.

Next, we show that $g^y_k \in \partial_y f(x_k,y_k)$. Indeed,
\begin{align*}
		f(x_k,y) - f(x_k,y_k) - \langle g^y_k, y-y_k \rangle
		 & =
		\sup_{z \in \mathcal Z} \und{\min}{i \in I} f_i + \langle z, x_k-x_i \rangle  + \langle g^y_i, y-y_i \rangle \\&\quad - f_k - \langle g^y_k, y-y_k \rangle
		\\&\leqslant
		\sup_{z \in \mathcal Z} f_k + \langle z, x_k-x_k \rangle + \langle g^y_k, y-y_k \rangle \\&\quad - f_k - \langle g^y_k, y-y_k \rangle
		\\&= 0.
\end{align*}
	Hence, $g^y_k \in \partial_y f(x_k,y_k)$ by the definition of the subdifferential.

	Finally, we show that function $f(x,y)$ satisfies Definition~\ref{fed:lip}. Let $(x_1,y_1), (x_2,y_2) \in \sX \times \sY$. Then, we get
	\begin{align*}
		f(x_1,y_1) - f(x_2,y_2)
		 & = \sup_{z_1 \in \mathcal Z} \und{\min}{i_1 \in I} f_{i_1} + \langle z_1, x_1-x_{i_1} \rangle + \langle g^y_{i_1}, y_1-y_{i_1} \rangle \\
		 &\quad - \sup_{z_2 \in \mathcal Z} \und{\min}{i_2 \in I} f_{i_2} + \langle z_2, x_2-x_{i_2} \rangle + \langle g^y_{i_2}, y_2-y_{i_2} \rangle
		\\&=
		\sup_{z_1 \in \mathcal Z}\inf_{z_2 \in \mathcal Z}\und{\min}{i_1 \in I}\hspace{4pt}\und{\max}{i_2 \in I}
		\left[
		f^{i_1} + \langle z_1, x_1-x_{i_1} \rangle + \langle g^y_{i_1}, y_1-y_{i_1} \rangle \right.
		\\&\left.\quad - f_{i_2} - \langle z_2, x_2-x_{i_2} \rangle - \langle g^y_{i_2}, y_2-y_{i_2} \rangle \right]
		\\&=
		\sup_{z_1 \in \mathcal Z}\inf_{z_2 \in \mathcal Z}\und{\max}{i_2 \in I}\hspace{4pt}\und{\min}{i_1 \in I}
		\left[
		f_{i_1} + \langle z_1, x_1-x_{i_1} \rangle + \langle g^y_{i_1}, y_1-y_{i_1}\rangle \right.
		\\&\left.\quad  - f_{i_2} - \langle z_2, x_2-x_{i_2} \rangle - \langle g^y_{i_2}, y_2-y_{i_2}\rangle \right]
		\\&\leqslant
		\sup_{z_1 \in \mathcal Z}\und{\max}{i_2 \in I}
		\left[
		f_{i_2} + \langle z_1, x_1-x_{i_2} \rangle + \langle g^y_{i_2}, y_1-y_{i_2} \rangle \right.
		\\&\left.\quad - f_{i_2} - \langle z_1, x_2-x_{i_2} \rangle - \langle g^y_{i_2}, y_2-y_{i_2}\rangle \right]
		\\&=
		\sup_{z_1 \in \mathcal Z}\und{\max}{i_2 \in I}
		\langle z_1, x_1-x_2\rangle + \langle g^y_{i_2}, y_1-y_2 \rangle
		\\&\leqslant
		\sup_{z_1 \in \mathcal Z}\und{\max}{i_2 \in I}
		\|z_1\|_2 \|x_1 - x_2\|_2 + \|g^y_{i_2}\|_2 \|y_1 - y_2\|_2
		\\&\leqslant
		M\left(\|x_1 - x_2\|_2 + \|y_1 - y_2\|_2\right),
	\end{align*}
	where we used \eqref{eq6} in the last inequality.
\end{proof}

\subsection{Difference of smooth convex functions}

Next, we solve the interpolation problem for several special cases of $\mathcal S_{\mu_x,\mu_y,L_x,L_y,L_{xy}}$. By letting $L_{xy} = 0$, one recovers difference of two smooth strongly convex functions:
\begin{align}
f(x,y) = p(x) - q(y),
\label{eq:separable}
\end{align}
where $p \in \mathcal F_{\mu_x, L_x}$ and $q \in \mathcal F_{\mu_y, L_y}$. 

We will require the interpolation conditions for smooth strongly convex functions that constitute the composites of $f$.

\begin{theorem} [$\mathcal F_{\mu, L}$-interpolability~\cite{taylor2017smooth}] \label{teo:SC_int} Sequence $\left\{(x_i, g_i, f_i)\right\}_{i \in I}$ is $\mathcal F_{\mu, L}$-interpolable if and only if the following inequality holds for all $i,j \in I$:
\begin{align}
    f_i &\geqslant f_j + \langle g_j,x_i-x_j \rangle \notag \\& \quad + \frac{L}{2(L-\mu)} \left(\frac{1}{L}\|g_i-g_j\|_2^2 + \mu \|x_i-x_j\|_2^2 - 2\frac{\mu}{L}\langle g_i-g_j, x_i-x_j \rangle\right).
\label{eq7}
\end{align}
\end{theorem}

\begin{lemma} \label{lem:sep_int}
Sequence $A_I$ is $\mathcal S_{\mu_x,\mu_y,L_x,L_y,0}$-interpolable if and only if there exist $\{p_i\}_{i\in I}: \ p_i \in \R, \ \forall i \in I$ such that the following inequalities are satisfied for all $i,j \in I$:
\begin{align*}
    p_i - p_j &\geqslant \langle g^x_j, x_i - x_j \rangle   \\
    &\quad+ \frac{L_x}{2(L_x-\mu_x)}\left(\frac{1}{L_x}\|g^x_i-g^x_j\|_2^2 - \frac{2\mu_x}{L_x}\langle g^x_i-g^x_j, x_i-x_j \rangle + \mu_x\|x_i-x_j\|_2^2\right), \\
    p_i - p_j &\geqslant f_i - f_j - \langle g^y_j, y_i - y_j \rangle  \\
    &\quad + \frac{L_y}{2(L_y-\mu_y)}\left(\frac{1}{L_y}\|g^y_i-g^y_j\|_2^2 + \frac{2\mu_y}{L_y}\langle g^y_i-g^y_j, y_i-y_j \rangle + \mu_y\|y_i-y_j\|_2^2\right).
\end{align*}
\end{lemma}
\begin{proof}
Necessity. Suppose $A_I$ can be interpolated by $p(x) - q(y)$. The inequalities follow from Theorem~\ref{teo:SC_int} where we additionally denoted $p_i := p(x_i)$. Sufficiency. Suppose $A_I$ satisfies the inequalities. By Theorem~\ref{teo:SC_int},  sequences $\left\{(x_i, g^x_i, p_i)\right\}_{i \in I}$ and $\left\{(y_i, -g^y_i,\ p_i - f_i)\right\}_{i \in I}$ are interpolable by the respective smooth strongly convex functions. Their difference interpolates $A_I$ which completes the proof.
\end{proof}

Because we treat interpolation of the special cases in unified manner, our next goal is to express the result of Lemma~\ref{lem:sep_int} in terms of the elements of $A_I$. First, we combine the inequalities from Lemma~\ref{lem:sep_int} to obtain
\begin{align}
    p_i - p_j \geqslant c^{ij} \quad \text{for} \quad  i,j \in I, 
\label{eq9}
\end{align}
where
\begin{align*}
    c^{ij} &= \max \left\{\alpha^{ij}, \beta^{ij} \right\}, \\
    \alpha^{ij} &= \langle g^x_j, x_i - x_j \rangle  \\&\quad+ \frac{L_x}{2(L_x-\mu_x)}\left(\frac{1}{L_x}\|g^x_i-g^x_j\|_2^2 - \frac{2\mu_x}{L_x}\langle g^x_i-g^x_j, x_i-x_j \rangle + \mu_x\|x_i-x_j\|_2^2\right), \\
    \beta^{ij} &= f_i - f_j - \langle g^y_j, y_i - y_j \rangle \\&\quad+ \frac{L_y}{2(L_y-\mu_y)}\left(\frac{1}{L_y}\|g^y_i-g^y_j\|_2^2 + \frac{2\mu_y}{L_y} \langle g^y_i-g^y_j, y_i-y_j \rangle + \mu_y\|y_i-y_j\|_2^2\right).
\end{align*}

Next, we derive necessary and sufficient conditions for feasibility of \ref{eq9}.

\begin{definition} $J := (j_0, \ldots , j_k)$ is an ordered selection of indices from $I$. Indices do not repeat except for $j_k$ that can can equal $j_0$.
\end{definition}

\begin{definition} \label{fed:index_select} $J(n)$ is a set of all $J$ such that $j_0 = 0$ and $j_k = n$, where $0, n \in I$.
\end{definition}

\begin{definition} \label{fed:index_sum}  Let $c^{ij} \in \R$ be real numbers indexed by $i,j \in I$ and $J \in J(n)$. Then $C_{J} := \overset{k-1}{\underset{l=0}{\sum}} c^{j_{l+1}j_{l}}$ is a sum of $c^{ij}$ corresponding to $J$.
\end{definition}

\begin{definition} \label{fed:maximal_sum} $C_{J(n)} := \und{\max}{J \in J(n)} C_{J}$ is the maximal sum.
\end{definition}

\begin{definition} \label{fed:cyclic_sum} The \textit{cyclic sum} of $c^{ij}$ that corresponds to $J$ is $    \overset{k}{\underset{l=0}{\sum}} c^{j_{l+1}j_{l}}$, where $j_{k+1} := j_0$.
\end{definition}

The following Lemma~\ref{lem:lin_ineq} was inspired by the concept of cyclic monotonicity introduced in classical works of R. Tyrrell Rockafellar~\cite{rockafellar1966characterization} and could be considered a variation of Theorem 3.2 from \cite{artstein2022rockafellar}.

\begin{lemma} A system of linear inequalities~\eqref{eq9} is feasible (with respect to $p \in \R^{|I|}$) if and only if all cyclic sums of $c^{ij}$ (Definition ~\ref{fed:cyclic_sum}) are non-positive.
\label{lem:lin_ineq}
\end{lemma}

\begin{proof}
Suppose $p \in \R^{|I|}$ satisfies the system of inequalities. Let $J = (j_0,...,j_{k})$ be an arbitrary ordered subset of $I$. By adding the following inequalities
\begin{align*}
    & p_{j_1} - p_{j_0} \geqslant c^{j_1j_0}, \\
    & p_{j_2} - p_{j_1} \geqslant c^{j_2j_1}, \\
    &... \\
    & p_{j_k} - p_{j_{k-1}} \geqslant c^{j_{k}j_{k-1}}, \\
    & p_{j_0} - p_{j_k} \geqslant c^{j_0j_k},
\end{align*}
we prove that the cyclic sum of $c^{ij}$ corresponding to $J$ is non-positive. 

Suppose all cyclic sums of $c^{ij}$ are non-positive. We will show that $p_i := C_{J(i)}$ $i \in I$, where $C_{J(i)}$ is a maximal sum (Definition ~\ref{fed:maximal_sum}), satisfy \eqref{eq9}: 
\begin{equation}
\und{\max}{J \in J(i)} C_{J} \geqslant \und{\max}{J \in J(j)} C_{J} + c^{ij}.
\label{eq:cyclic_sums}
\end{equation}
The inequalities obviously hold for $i=j$, so consider $i \neq j$. There are two possible cases.

First case: $\overline{J}_j = \und{\text{argmax}}{\overline{J} \in J(j)} \hspace{2pt} C_{\overline{J}}$ \ does not contain $i$. Then $(\overline{J}_j, i) \in J(i)$ and the sum of $c^{ij}$ corresponding to $(\overline{J}_j, i)$ is exactly $\und{\max}{J \in J(j)} C_{J} + c^{ij}$. This sum however is not bigger than $\und{\max}{J \in J(i)} C_{J}$ that is the maximal sum for $J(i)$.

Second case: $\overline{J}_j = \und{\text{argmax}}{\overline{J} \in J(j)} \hspace{2pt} C_{\overline{J}}$ \quad contains $i$. From Definition~\ref{fed:index_select} we have that 
$\overline{J}_j \in J(j)$ can contain $i$ only once:
\begin{align*}
    \overline{J}_j = (j_0 = 0, ..., \mathbf{j}_m = \mathbf{i}, ..., j_k = j)
\end{align*}
The sum of $c^{ij}$ (Definition~\ref{fed:index_sum}) corresponding to $(\overline{J}_j, i)$ can be written as
\begin{align*}
    C_{(\overline{J}_j, i)} = C_{\overline{J}_{j1}} + C_{\overline{J}_{j2}}.
\end{align*}
where
\begin{align*}
    &\overline{J}_{j1} = (j_0 = 0, ..., j_m = i), \\
    &\overline{J}_{j2} = (j_m = i, ..., j_{k+1} = i).
\end{align*}
By definition, $\overline{J}_{j1} \in J(i)$. And because any cyclic sum of $c^{ij}$ is non-positive, we have $C_{\overline{J}_{j2}} \leqslant 0$. Therefore:
\begin{align*}
    &\und{\max}{J \in J(j)} C_{J} + c^{ij} = C_{(\overline{J}_j, i)} = C_{\overline{J}_{j1}} + C_{\overline{J}_{j2}} \leqslant C_{\overline{J}_{j1}} \leqslant \und{\max}{J \in J(i)} C_{J}.
\end{align*}
which is exactly \eqref{eq:cyclic_sums}.
\end{proof}

With Lemma~\ref{lem:lin_ineq}, we can rewrite the interpolation conditions in the desired form. 

\begin{theorem} \label{teo:sep_int} Let $\forall i,j \in I: \ c^{ij}:= \max \left(\alpha^{ij}, \beta^{ij} \right)$, where
\begin{align*}
    \alpha^{ij} &= \langle g^x_j, x_i - x_j \rangle \\&\quad+ \frac{L_x}{2(L_x-\mu_x)}\left(\frac{1}{L_x}\|g^x_i-g^x_j\|_2^2 - \frac{2\mu_x}{L_x}\langle g^x_i-g^x_j,x_i-x_j \rangle+ \mu_x\|x_i-x_j\|_2^2\right), \\
    \beta^{ij} &= f_i - f_j - \langle g^y_j, y_i - y_j \rangle \\&\quad+ \frac{L_y}{2(L_y-\mu_y)}\left(\frac{1}{L_y}\|g^y_i-g^y_j\|_2^2 + \frac{2\mu_y}{L_y} \langle g^y_i-g^y_j, y_i-y_j \rangle + \mu_y\|y_i-y_j\|_2^2\right).
\end{align*}
Sequence $A_I$ is $\mathcal S_{\mu_x,\mu_y,L_x,L_y,0}$-interpolable if and only if any cyclic sum (Definition ~\ref{fed:cyclic_sum}) of $c^{ij}$ is non-positive.
\end{theorem}
\begin{proof}
Apply Lemma~\ref{lem:lin_ineq} to the system of inequalities \eqref{eq9}.
\end{proof}

\subsection{Bilinear functions}

We begin with the family of matrix game functions
\begin{equation}
    f(x,y) = y^\top A x,
\label{problem:bilinear}
\end{equation}
where $A \in \R^{d_y \times d_x}$ is a coupling matrix with bounded largest singular value $\sigma_{\max}(A) \leqslant L$. We denote $\mathcal L_L = \{A\colon \ \sigma_{\max}(A) \leqslant L \}$ and $f \in \mathcal B_L$.

\begin{definition}[$\mathcal L_L$ - interpolation~\cite{bousselmi2024interpolation}] Sets of pairs $\left\{(x_i, y_i)\right\}_{i\in[N_1]}$ and $\left\{(u_j, v_j)\right\}_{j\in[N_2]}$ are $\mathcal L_L$-interpolable if and only if $\exists A \in \mathcal L_L$ such that
\begin{align*}
    \begin{cases*}
        y_i = Ax_i \quad \forall i \in [N_1], \\
        v_j = A^\top u_j \quad \forall j \in [N_2].
    \end{cases*}
\end{align*}
\label{fed:lin_int}
\end{definition}

\begin{theorem} [$\mathcal L_L$ - interpolation~\cite{bousselmi2024interpolation}]. Let $X \in \R^{n\times N_1}$, $Y \in \R^{m\times N_1}$, $U \in \R^{m\times N_2}$, $V \in \R^{n\times N_2}$, and $0 \leqslant L < \infty$. $(X,Y,U,V)$ is $\mathcal L_L$-interpolable if and only if
\begin{align*}
    \begin{cases*}
        X^\top V = Y^\top U, \\
        Y ^\top Y \preceq L^2 X^\top X, \\
        V ^\top V \preceq L^2 U^\top U.
    \end{cases*}
\end{align*}
\label{teo:lin_int}
\end{theorem}

Using the above result, we can derive the interpolation conditions for $\mathcal B_L$.

\begin{theorem} \label{teo:bil_int} Sequence $A_I$ is $\mathcal B_L$-interpolable if and only if
\begin{align*}
    \begin{cases*}
        f_i = \langle g^x_i, x_i \rangle = \langle g^y_i, y_i \rangle \quad \forall i \in I, \\
        G_x^\top X = Y^\top G_y, \\
        G_x ^\top G_x \preceq L^2 Y^\top Y, \\
        G_y ^\top G_y \preceq L^2 X^\top X.
    \end{cases*}
\end{align*}
where $X, Y, G_x, G_y$ are matrices that have $\{x_i\}_{i\in I}$, $\{y_i\}_{i\in I}$, $\{g^x_i\}_{i\in I}$, $\{g^y_i\}_{i\in I}$ as columns respectively.
\end{theorem}

\begin{proof}
Suppose $A_I$ can be interpolated by $f \in \mathcal B_L$. From \eqref{problem:bilinear}, we have 
\begin{align*}
    &g^x_i = \nabla_x f(x_i, y_i) = A^\top y_i, \\
    &g^y_i = \nabla_y f(x_i, y_i) = A x_i.
\end{align*}
Therefore, by Theorem~\ref{teo:lin_int}:
\begin{align*}
    &G_x^\top X = Y^\top G_y, \\
    &G_x ^\top G_x \preceq L^2 Y^\top Y, \\
    &G_y ^\top G_y \preceq L^2 X^\top X.
\end{align*}
Next, suppose the conditions are satisfied. With Theorem~\ref{teo:lin_int} we can find a linear operator $A$ that interpolates $(X,\ G_y,\ Y,\ G_x)$. We have to verify that $\langle y, A x \rangle$ has correct values in every $(x_i,y_i)$:
\begin{align*}
    f(x_i,y_i) = \langle y_i, A x_i \rangle = \langle g^x_i, x_i \rangle = \langle g^y_i, y_i \rangle = f_i.
\end{align*}
This proves that $\langle y, A x \rangle$ interpolates $A_I$ correctly.
\end{proof}

Bilinear functions $\mathcal S_{0,0,0,0,L_{xy}}$ can be recovered by letting $\mu_x=\mu_y=L_x=L_y=0$:
\begin{equation}
    f(x,y) = y^\top A x + a^\top x + b^\top y,
\label{problem:gen_bilinear}
\end{equation}
where $A \in \mathcal L_{L_{xy}}$, $a \in \sX$, and $b \in \sY$.

Bilinear problems rarely play important role in practice. Instead, they are often used for theoretical purposes and preliminary analyses \cite{azizian2020accelerating, zhang2022near, gidel2019negative, mokhtari2020unified, daskalakis2017training, liang2019interaction}.

\begin{lemma} If $A_I$ satisfies $\forall i,j \in I$
\begin{equation}
    f_i = f_j + \langle g^x_j, x_i - x_j \rangle + \langle g^y_i, y_i - y_j \rangle 
\label{eq:cyclic_grads_1}
\end{equation}
then $\forall i,j,k,m \in I$, it holds that
\begin{equation*}
    \langle g^x_i - g^x_k, x_j - x_m \rangle = \langle g^y_j - g^y_m, y_i - y_k \rangle.
\label{eq:cyclic_grads_2}
\end{equation*}
\label{lem:cyclic_grads}
\end{lemma}
\begin{proof}
We write inequalities \eqref{eq:cyclic_grads_1} in cyclic manner and add them up. For cycles of length $2$,
\begin{align*}
    &f_i = f_j + \langle g^x_j, x_i - x_j \rangle + \langle g^y_i, y_i - y_j \rangle, \\
    &f_j = f_i + \langle g^x_i, x_j - x_i \rangle + \langle g^y_j, y_j - y_i \rangle,
\end{align*}
we obtain
\begin{align}
    \langle g^x_i - g^x_j, x_i - x_j \rangle = \langle g^y_i - g^y_j, y_i - y_j \rangle.
\label{eq:cyclic_grads_3}
\end{align}
After doing the same for a cycle of length $3$ and combining with \eqref{eq:cyclic_grads_3}, we get
\begin{align}
    \langle g^x_i - g^x_k, x_j - x_k \rangle = \langle g^y_j - g^y_k, y_i - y_k \rangle.
\label{eq:cyclic_grads_4}
\end{align}
Finally, for a cycle of length $4$ and combining with both \eqref{eq:cyclic_grads_3} and \eqref{eq:cyclic_grads_4}, we obtain \eqref{eq:cyclic_grads_2}.

\end{proof}

The proof of the next Theorem~\ref{teo:gen_bilinear} employs similar argument to Theorem~\ref{teo:bil_int}.

\begin{theorem} \label{teo:gen_bilinear} Sequence $A_I$ is $\mathcal S_{0,0,0,0,L_{xy}}$-interpolable if and only if $\forall i, j \in I$:
\begin{align*}
    f_i = f_j + \langle g^x_j, x_i - x_j \rangle + \langle g^y_i, y_i - y_j \rangle,
\end{align*}
and $\forall p \in I$:
\begin{align*}
    \begin{cases*}
        G_x^p {}^\top G_x^p \preceq L_{xy}^2 Y^p{}^\top Y^p,\\
        G_y^p {}^\top G_y^p \preceq L_{xy}^2 X^p{}^\top X^p,
    \end{cases*}
\end{align*}
where $X^p, Y^p, G_x^p, G_y^p$ are matrices that have $\left\{x_i-x_p\right\}_{i\in I}$, $\left\{y_i-y_p\right\}_{i\in I}$, \\ $\left\{g^x_i-g^x_p\right\}_{i\in I}$, and $\left\{g^y_i-g^y_p\right\}_{i\in I}$ as columns respectively.
\end{theorem}

\begin{proof}
The necessity can be verified straightforwardly. Next, suppose that the conditions are satisfied. With Lemma~\ref{lem:cyclic_grads}, we obtain $\forall p, k, m \in I$:
\begin{align*}
    &G_y^m{}^\top Y^k = X^m{}^\top G_x^k, \\
    &G_x^p{}^\top G_x^p \preceq L_{xy}^2 Y^p{}^\top Y^p, \\
    &G_y^p{}^\top G_y^p \preceq L_{xy}^2 X^p{}^\top X^p.
\end{align*}
By Theorem~\ref{teo:lin_int}, for any $p \in I$, there exists $A_p \in \mathcal L_{L_{xy}}$ such that $\forall i \in I$:
\begin{align*}
    &g^x_i - g^x_p = A_p(y_i - y_p), \\
    &g^y_i - g^y_p = A_p^\top (x_i - x_p).
\end{align*}
We choose matrix $A := A_p$, corresponding to some particular $p$, for example $p = 0$, and construct a function:
\begin{align*}
    F(x,y) &= \langle x, A y \rangle + \langle g^x_0 - Ay_0, x \rangle + \langle g^y_0 - A ^\top x_0,y\rangle \\&\quad+ f_0 - \langle g^x_0,x_0\rangle - \langle g^y_0,y_0 \rangle + \langle x_0, A y_0 \rangle.
\end{align*}
Now we show that $F(x,y)$ correctly interpolates ${(x_i, y_i, g^x_i, g^y_i, f_i)}_{i \in I}$:
\begin{align*}
    F(x_i, y_i) &= f_0 + \langle g^x_0,x_i - x_0\rangle + \langle x_i, A(y_i - y_0) \rangle - \langle x_0, A(y_i - y_0) \rangle + \langle g^y_0, y_i - y_0 \rangle \\&= f_0 + \langle g^x_0, x_i - x_0 \rangle + \langle g^y_i, y_i - y_0 \rangle \\&= f_i, \\
    \nabla_xF(x_i,y_i) &= Ay_i - Ay_0 + g^x_0 \\ &= A(y_i-y_0) + g^x_0 \\&= g^x_i, \\
    \nabla_y F(x_i,y_i) &= A^\top x - A^\top x_0 + g^y_0 \\&= g^y_i.
\end{align*}
This concludes the proof.

\end{proof}

\subsection{Convex--concave functions with bilinear coupling}

Consider composite functions of the form
\begin{align}
     f(x,y) = p(x) + y^\top A x- q(y),
\label{problem:bilinear_coupling}
\end{align}
where $p \in \mathcal F_{\mu_x,L_x}$, $q \in \mathcal F_{\mu_y,L_y}$, and $A \in \mathcal L_{L_{xy}}$. We denote the class by $\mathcal K_{\mu_x,\mu_y,L_x,L_y,L_{xy}}$. It's straightforward to verify that $\mathcal K_{\mu_x,\mu_y,L_x,L_y,L_{xy}} \subset \mathcal S_{\mu_x,\mu_y,L_x,L_y,L_{xy}}$.

Because we know interpolation conditions of the individual composites, derivation of the interpolation conditions is straightforward.

\begin{theorem}
Sequence $A_I$ is $\mathcal K_{\mu_x,\mu_y,L_x,L_y,L_{xy}}$-interpolable if and only if $\forall i \in I$, there exist $h^x_i, \phi^x_i \in \sX$, $h^y_i, \phi^y_i \in \sY$ and $f^x_i, f^y_i \in \R$ such that 
\begin{align*}
    &f_i = f^x_i + \langle h^x_i, x_i \rangle - f^y_i, \\
    &g^x_i = \phi^x_i + h^x_i, \\
    &g^y_i = h^y_i - \phi^y_i, \\
    &f^x_i \geqslant f^x_j + \langle \phi^x_j,x_i-x_j\rangle \\&\quad+ \frac{L_x}{2(L_x-\mu_x)} \left(\frac{1}{L_x}\|\phi^x_i-\phi^x_j\|_2^2 + \mu_x \|x_i-x_j\|_2^2 - 2\frac{\mu_x}{L_x}\langle \phi^x_i-\phi^x_j, x_i-x_j\rangle\right), \\
    &f^y_i \geqslant f^y_j + \langle \phi^y_j, y_i-y_j \rangle \\&\quad+ \frac{L_y}{2(L_y-\mu_y)} \left(\frac{1}{L_y}\|\phi^y_i-\phi^y_j\|_2^2 + \mu_y \|y_i-y_j\|_2^2 - 2\frac{\mu_y}{L_y}\langle \phi^y_i-\phi^y_j, y_i-y_j \rangle\right), \\
    &H_x^\top X = Y^\top H_y, \\
    &H_x ^\top H_x \preceq L_{xy}^2 Y^\top Y, \\
    &H_y ^\top H_y \preceq L_{xy}^2 X^\top X.
\end{align*}
where $X, Y, H_x, H_y$ are matrices that have $\{x_i\}_{i\in I}$, $\{y_i\}_{i\in I}$, $\{h^x_i\}_{i\in I}$, $\{h^y_i\}_{i\in I}$ as columns respectively.
\label{teo:bil_coupling_int}
\end{theorem}

\begin{proof}
The statement of the theorem follows from combining Theorem~\ref{teo:SC_int} and Theorem~\ref{teo:bil_int}. The logic of the proof is straightforward and similar to Lemma~\ref{lem:sep_int} (interpolation conditions for difference of strongly convex functions).

\end{proof}

\section{Construction of PEP for composite saddle-point problem} \label{sec:pep_construction}

We construct an SDP program to solve PEP for first-order methods that find a saddle point of  $f \in \mathcal K_{\mu_x,\mu_y,L_x,L_y,L_{xy}}$. Consider a family of fixed-step first-order iterative methods:
\begin{align*}
    &u^x_{k} = \overset{n}{\underset{j=0}{\sum}} \theta^x_j \hspace{2pt} x_{k-j}, \quad
    v^y_{k} = \overset{n}{\underset{j=0}{\sum}} \phi^y_j \hspace{2pt} y_{k-j}, \label{method:M}\tag{$\mathcal A$}\\
    &g^x_{k} = \nabla_x p(u^x_{k}), \quad
    h^x_{k} = A^\top v^y_{k}, \\
    &x_{k+1} = \overset{n}{\underset{j=0}{\sum}} \alpha^x_{j} \hspace{2pt} x_{k-j} - \beta^x \hspace{2pt} g^x_{k} -  \gamma^x \hspace{2pt} h^x_{k}, \\
    &u^y_{k} = \overset{n}{\underset{j=0}{\sum}} \theta^y_j \hspace{2pt} y_{k-j}, \quad
    v^x_{k} = \overset{n+1}{\underset{j=0}{\sum}} \phi^x_j \hspace{2pt} x_{k+1-j}, \\
    &g^y_{k} = \nabla_y q(u^y_{k}), \quad
    h^y_{k} = A v^x_{k}, \\
    &y_{k+1} = \overset{n}{\underset{j=0}{\sum}} \alpha^y_{j} \hspace{2pt} y_{k-j} -  \beta^y \hspace{2pt} g^y_{k} +  \gamma^y \hspace{2pt} h^y_{k},
\end{align*}
where $A$ is a coupling matrix from \eqref{problem:bilinear_coupling} and $n$ is degree (or memory) of the algorithm. The algorithm should have a fixed point $z_\ast = (x_\ast, y_\ast)$ which necessitates
\begin{align}
    \overset{n}{\underset{j=0}{\sum}} \theta^x_j =
    \overset{n}{\underset{j=0}{\sum}} \theta^y_j =
    \overset{n+1}{\underset{j=0}{\sum}} \phi^x_j =
    \overset{n}{\underset{j=0}{\sum}} \phi^y_j =
    \overset{n}{\underset{j=0}{\sum}} \alpha^x_{j} = \overset{n}{\underset{j=0}{\sum}} \alpha^y_{j} = 1
    \label{method_valid_cond}
\end{align}
Most known accelerated schemes can be recovered by performing (at most) two iterations of \eqref{method:M} with different sets of constants. Following standard PEP logic, we introduce a finite dimensional formulation of PEP. All the constraints are expressed in terms of oracle's and algorithmic outputs. The same goes for variables. Suppose \eqref{method:M} iterates for $k = 0,\ldots,N$ ($N+1$ iterations) and denote the set of indexes: $I := \left\{0,\ldots ,N,\ast \right\}$.
\begin{align} \label{problem:finite_PEP}\tag{finite-spp-PEP}
    &\und{\sup}{\{x_i, y_i, h^x_i, h^y_i, g^x_i, g^y_i, f^x_i, f^y_i\}_{i \in I}} \mathcal P(\{x_i, y_i, h^x_i, h^y_i, g^x_i, g^y_i, f^x_i, f^y_i\}_{i \in I})  \\
    &\text{subject to } \nonumber \\
    &\quad \quad\quad \{x_i, y_i,u^x_i,u^y_i,v^x_i,v^y_i\}^{N}_1 \text{ are generated by method \eqref{method:M}}, \nonumber \\
    &\quad \quad\quad\left\{(u^x_i, g^x_i, f^x_i)\right\}_{i \in I} \text{ is } \mathcal F_{\mu_x,L_x}-\text{interpolable}, \nonumber \\
    &\quad \quad\quad\left\{(u^y_i, g^y_i, f^y_i)\right\}_{i \in I} \text{ is } \mathcal F_{\mu_y,L_y}-\text{interpolable}, \nonumber \\
    &\quad \quad\quad\left\{(v^y_i,h^x_i)\right\}_{i \in I}, \{(v^x_i,h^y_i)\}_{i \in I} \text{ is } \mathcal L_{L_{xy}}-\text{interpolable}, \nonumber \\
    &\quad \quad\quad \left(x_\ast,g^x_\ast, h^x_\ast, y_\ast, g^y_\ast, h^y_\ast\right) = \left(\textbf{0}_{d_x}, \textbf{0}_{d_x}, \textbf{0}_{d_x}, \textbf{0}_{d_y}, \textbf{0}_{d_y}, \textbf{0}_{d_y}\right), \nonumber \\
    &\quad\quad\quad \|x_0-x_\ast\|_2^2 + \|y_0-y_\ast\|_2^2 \leqslant R^2. \nonumber
\end{align}
Because, we use necessary and sufficient interpolation conditions, the finite dimensional formulation is equivalent to the original formulation in terms of optimal value \cite{taylor2017smooth}.

Because $\mathcal K_{\mu_x,\mu_y,L_x,L_y,L_{xy}}$ is invariant to translations, we let $(x_*,y_*) = (\textbf{0}_{d_x}, \textbf{0}_{d_y})$. From the problem being unconstrained, we have $g^x_* = -h^x_* = \textbf{0}_{d_x}$ and $g^y_* = h^y_* = \textbf{0}_{d_y}$. By shifting function's values, we can set $f^x_* = f^y_* = 0$. We aim to rewrite \eqref{problem:finite_PEP} as a convex semidefinite program. First, we initialize row basis vectors for initial conditions, gradient values and functional values $\overline{x}_{k}, \overline{g}^{x}_{k}, \overline{h}^{x}_{k} \in \R^{2N+3}$, 
$\overline{y}_{k}, \overline{g}^{y}_{k}, \overline{h}^{y}_{k} \in \R^{2N+3}$, $\overline{f}^{x}_{k} \in \R^{N+1}, \overline{f}^{y}_{k} \in \R^{N+1}$:
\begin{align*}
    &\overline{x}_{k} = \overline{y}_{k}:= \overline{e}_{1}^\top \quad k \in {-n,..,0}, \\
    &\overline{g}^{x}_{k} = \overline{g}^{y}_{k} := \overline{e}_{k+2}^\top \quad k \in {0,..,N}, \\ 
    &\overline{h}^{x}_{k} = \overline{h}^{y}_{k} := \overline{e}_{k+N+3}^\top \quad k \in {0,..,N}, \\
    &\overline{f}^{x}_{k} = \overline{f}^{y}_{k} := \overline{e}_{k+1}^\top \quad k \in {0,..,N}, \\
    &\overline{x}_{\ast} = \overline{y}_{\ast} = \overline{g}^{x}_{\ast} = \overline{g}^{y}_{\ast} = \overline{h}^{x}_{\ast} = \overline{h}^{y}_{\ast} := \mathbf{0}_{ 2N+3}^\top, \\  
    &\overline{f}^{x}_{\ast} = \overline{f}^{y}_{\ast} := \mathbf{0}_{N+1}^\top, 
\end{align*}
where $\overline{e}_{i}$ is the $i^{\text{th}}$ unit vector. The rest of row vectors are calculated in accordance with the algorithm:
\begin{align*}
    &\overline{u}^{x}_{k} = \overset{n}{\underset{j=0}{\sum}} \theta^x_j \hspace{2pt} \overline{x}_{k-j}, \quad \overline{v}^{y}_{k} = \overset{n}{\underset{j=0}{\sum}} \phi^y_j \hspace{2pt} \overline{y}_{k-j}, \\
    &\overline{u}^{y}_{k} = \overset{n}{\underset{j=0}{\sum}} \theta^y_j \hspace{2pt} \overline{y}_{k-j}, \quad \overline{v}^{x}_{k} = \overset{n+1}{\underset{j=0}{\sum}} \phi^x_j \hspace{2pt} \overline{x}_{k+1-j}, \\
    &\overline{x}_{k+1} = \overset{n}{\underset{j=0}{\sum}} \alpha^x_{j} \hspace{2pt} \overline{x}_{k-j} -  \beta^x \hspace{2pt} \overline{g}^{x}_{k} -  \gamma^x \hspace{2pt} \overline{h}^{x}_{k}, \\
    &\overline{y}_{k+1} = \overset{n}{\underset{j=0}{\sum}} \alpha^y_{j} \hspace{2pt} \overline{y}_{k-j} -  \beta^y \hspace{2pt} \overline{g}^{y}_{k} +  \gamma^y \hspace{2pt} \overline{h}^{y}_{k}.
\end{align*}
We introduce Gram matrices $G_x = W_x^\top W_x \in \SM^{2N+3}$, $G_y= W_y^\top W_y \in \SM^{2N+3}$ where
\begin{align*}
    &W_x := \left[x_0 \quad g^x_{0} \quad \ldots \quad g^x_{N} \quad h^x_{0} \quad \ldots \quad h^x_{N}\right], \quad
    W_y := \left[y_0 \quad g^y_{0} \quad ... \quad g^y_{N} \quad h^y_{0} \quad ... \quad h^y_{N}\right],
\end{align*}
and also vectors of function values: 
\begin{align*}
    &\textbf{f}^x := \left[ f^x_{0} \ldots  f^x_N\right]^\top, \quad
    \textbf{f}^y := \left[ f^y_{0}, \ldots , f^y_N\right]^\top.
\end{align*}
Next, we express the interpolation conditions in terms of Gram matrices. With the Gram matrices introduced earlier, the interpolation conditions for strongly convex functions (Theorem~\ref{teo:SC_int}) $f(x) \in \mathcal F_{\mu_x,L_x}$, and $g(y) \in \mathcal F_{\mu_y,L_y}$ can be written as
\begin{align*}
    &0 \leqslant m^{x}_{ij}{}^\top\textbf{f}^x + \text{tr}(G_x M^{x}_{ij}) \quad \forall i,j \in I, \\
    &0 \leqslant m^{y}_{ij}{}^\top\textbf{f}^y + \text{tr}(G_y M^{y}_{ij}) \quad \forall i,j \in I,
\end{align*}
where matrices $m^{x}_{ij}, m^{y}_{ij} \in R^{N+1}$ and $M^{x}_{ij}, M^{y}_{ij} \in \SM^{2N+3}$ are defined as
\begin{align*}
    &m^{x}_{ij} := (L_x-\mu_x)(\overline{f}^{x}_{i} -\overline{f}^{x}_{j})^\top, \quad m^{y}_{ij} := (L_y-\mu_y)(\overline{f}^{y}_{i} -\overline{f}^{y}_{j})^\top, \\
    &M^{x}_{ij} := 
    \begin{bmatrix} 
        \overline{u}^{x}_i\\ \overline{u}^{x}_j\\ \overline{g}^{x}_i\\ 
        \overline{g}^{x}_j
    \end{bmatrix}^\top
    M^x
    \begin{bmatrix} 
        \overline{u}^{x}_i\\ \overline{u}^{x}_j\\ \overline{g}^{x}_i\\ 
        \overline{g}^{x}_j
    \end{bmatrix} ,\quad
    M^{y(K)}_{ij} := 
    \begin{bmatrix} 
        \overline{u}^{y}_i\\ \overline{u}^{y}_j\\ \overline{g}^{y}_i\\ 
        \overline{g}^{y}_j
    \end{bmatrix}^\top
    M^y
    \begin{bmatrix} 
        \overline{u}^{y}_i\\ \overline{u}^{y}_j\\ \overline{g}^{y}_i\\ 
        \overline{g}^{y}_j
    \end{bmatrix},
\end{align*}
where $M^x, M^y \in \SM^4$:
\begin{align}
    &M^x := M(\mu_x,L_x),\quad M^y := M(\mu_y,L_y), \label{eq:interp_matrix} \\ \nonumber \\
	&M(\mu,L) := \frac{1}{2} \begin{bmatrix*}[r]
	-\mu L  & \mu L & \mu & -L \\
	\mu L & -\mu L & -\mu & L \\
	\mu  & -\mu & -1 & 1 \\
	-L & L & 1 & -1 
    \end{bmatrix*}, \nonumber
\end{align}
Conditions for a bilinear function (Theorem~\ref{teo:bil_int}) can be rewritten as
\begin{align*}
    & B_{Hx}^\top G_x B_{vx} = B_{vy}^\top G_y B_{Hy}, \\
    &B_{Hx}^\top G_x B_{Hx} \preceq L_{xy}^2 B_{vy}^\top G_y B_{vy}, \\
    &B_{Hy}^\top G_y B_{Hy} \preceq L_{xy}^2 B_{vx}^\top G_xB_{vx},
\end{align*}
where $B_{Hx}, B_{Hy} \in \R^{(2N+3)\times(N+1)}$ and $B_{vx}, B_{vy} \in \R^{(2N+3)\times(N+1)}$ are defined as 
\begin{align*}
    &B_{Hx} = \left[(\overline{h}^{x}_0)^\top, \quad \ldots, \quad (\overline{h}^{x}_N)^\top\right], \quad B_{Hy} = \left[(\overline{h}^{y}_0)^\top, \quad \ldots, \quad (\overline{h}^{y}_N)^\top\right], \\
    &B_{vx} = \left[(\overline{v}^{x}_{0})^\top, \quad..., \quad (\overline{v}^{x}_{N})^\top\right], \quad
    B_{vy} = \left[(\overline{v}^{y}_{0})^\top, \quad ..., \quad (\overline{v}^{y}_{N})^\top\right].
\end{align*}
And lastly, performance metric
\begin{align*}
    \mathcal P := \text{tr}(G_x P_x) + \text{tr}(G_y P_y) + p_x^\top \textbf{f}^x + p_y^\top \textbf{f}^y 
\end{align*}
After expressing constraints and performance metric in terms of Gram matrices, we can reformulate (or, \textit{lift}) \eqref{problem:finite_PEP} so that optimization is carried out over $G_x, G_y, \textbf{f}^x, \textbf{f}^y$.
\begin{align*} \label{problem:SDP_PEP}\tag{sdp-spp-PEP}
    &\und{\sup}{G_x, G_y, \textbf{f}^x, \textbf{f}^y} \text{tr}(G_x P_x) + \text{tr}(G_y P_y) + p_x^\top \textbf{f}^x + p_y^\top \textbf{f}^y  \\
    &\text{subject to} \\
    &\quad\quad\quad m^{x}_{ij}{}^\top \textbf{f}^x + \text{tr}(G_x M^{x}_{ij}) \geqslant 0, \quad \forall i,j \in I,\\
    &\quad\quad\quad m^{y}_{ij}{}^\top \textbf{f}^y + \text{tr}(G_y M^{y}_{ij}) \geqslant 0, \quad \forall i,j \in I,\\
    &\quad\quad\quad B_{Hx}^\top G_x B_{vx} - B_{vy}^\top G_y B_{Hy} = 0, \\
    &\quad\quad\quad B_{Hx}^\top G_x B_{Hx} - L_{xy}^2 B_{vy}^\top G_y B_{vy} \preceq 0, \\
    &\quad\quad\quad B_{Hy}^\top G_y B_{Hy} - L_{xy}^2 B_{vx}^\top G_x B_{vx}\preceq 0, \\
    &\quad\quad\quad \text{tr}(G_xM^x_I) + \text{tr}(G_y M^y_I) - R^2 \leqslant 0.
\end{align*}

If there exists one-to-one correspondence between matrices $G_x$, $G_y$ and functions from $\mathcal K_{\mu_x,\mu_y,L_x,L_y,L_{xy}}$ then the optimal value of \eqref{problem:SDP_PEP} is equal to the optimal value of \eqref{problem:finite_PEP}. The correspondence can be guaranteed if dimensions satisfy $\min(d_x,d_y) \geqslant 2N+3$, which allows the matrices to be full rank \cite{taylor2017smooth}. In the case of $d_x < 2N+3$ or $d_y < 2N+3$, the optimal value equivalence will still hold under the additional rank constraints
\begin{align}
        &\text{Rank}(G_x) \leqslant d_x, \quad \text{Rank}(G_y) \leqslant d_y.
\end{align}
However, the rank constraints render the program non-convex. For large dimensional problems, one can assume that the condition $\min(d_x,d_y) \geqslant 2N+3$ is satisfied. In that case, the optimum of \eqref{problem:SDP_PEP} is attained and finite provided the constants $L_x, L_y, L_{xy}$ are bounded. This is because starting from bounded $(x_0,y_0)$, all other iterates and gradients are bounded as well due to algorithmic constraints and smoothness of $\mathcal K_{\mu_x,\mu_y,L_x,L_y,L_{xy}}$.

Any primal feasible solution of \eqref{problem:SDP_PEP}  provides a lower bound on the worst-case performance. To obtain analytical upper bounds, it's more convenient to work with the Lagrangian dual of \eqref{problem:SDP_PEP}. The feasible solutions of the dual program constitute  upper bounds on the worst-case behavior \cite{goujaud2023fundamental, taylor2017smooth}.

\section{Conclusion}

We derived necessary and sufficient interpolation conditions for classes of convex--concave functions and applied PEP framework to composite saddle-point problem. The interpolation problem was solved for general convex--concave functions and special cases of smooth class $\mathcal S_{\mu_x,\mu_y,L_x,L_y,L_{xy}}$. Interpolation by general smooth class appears to be  considerably more involved and is left for future work. For PEP, we considered both the standard formulation and automatic generation of quadratic Lyapunov functions.

\section{Acknowledgements}

The research is supported by the Ministry of Science and Higher Education of the Russian Federation (Goszadaniye), project No. FSMG-2024-0011

\bibliographystyle{plain}

\bibliography{new_references}

@book{bauschke2011convex,
  title={Convex Analysis and Monotone Operator Theory in Hilbert Spaces},
  author={Bauschke, Heinz H and Combettes, Patrick L},
  year={2011},
  publisher={Springer},
  address={Berlin},
  series={CMS Books in Mathematics},
  isbn={978-1-4419-9466-0}
}

@book{rockafellar1970,
  author    = {Rockafellar, R. Tyrrell},
  title     = {Convex Analysis},
  publisher = {Princeton University Press},
  year      = {1970},
  isbn      = {9780691015866},
  url       = {http://www.jstor.org/stable/j.ctt14bs1ff}
}

@article{taylor2017smooth,
  title     = {Smooth strongly convex interpolation and exact worst-case
               performance of first-order methods},
  author    = {Taylor, Adrien B and Hendrickx, Julien M and Glineur, Fran{\c c}ois},
  journal   = {Math. Program.},
  volume    = {161},
  number    = {1--2},
  pages     = {307--345},
  year      = {2017},
  doi       = {10.1007/s10107-016-1009-3}
}

@article{rockafellar1966characterization,
  title   = {Characterization of the subdifferentials of convex functions},
  author  = {Rockafellar, Ralph},
  journal = {Pacific Journal of Mathematics},
  volume  = {17},
  number  = {3},
  pages   = {497--510},
  year    = {1966},
  publisher = {Mathematical Sciences Publishers}
}

@article{artstein2022rockafellar,
  title   = {A Rockafellar-type theorem for non-traditional costs},
  author  = {Artstein-Avidan, Shiri and Sadovsky, Shay and Wyczesany, Katarzyna},
  journal = {Advances in Mathematics},
  volume  = {395},
  pages   = {108157},
  year    = {2022},
  publisher = {Elsevier}
}

@article{bousselmi2024interpolation,
  title   = {Interpolation conditions for linear operators and applications to performance estimation problems},
  author  = {Bousselmi, Nizar and Hendrickx, Julien M and Glineur, François},
  journal = {SIAM Journal on Optimization},
  volume  = {34},
  number  = {3},
  pages   = {3033--3063},
  year    = {2024},
  publisher = {SIAM}
}

@inproceedings{azizian2020accelerating,
  title     = {Accelerating smooth games by manipulating spectral shapes},
  author    = {Azizian, Waiss and Scieur, Damien and Mitliagkas, Ioannis and Lacoste-Julien, Simon and Gidel, Gauthier},
  booktitle = {Proceedings of the 23rd International Conference on Artificial Intelligence and Statistics},
  pages     = {1705--1715},
  year      = {2020},
  publisher = {PMLR}
}

@inproceedings{zhang2022near,
  title     = {Near-optimal local convergence of alternating gradient descent-ascent for minimax optimization},
  author    = {Zhang, Guodong and Wang, Yuanhao and Lessard, Laurent and Grosse, Roger B},
  booktitle = {Proceedings of the 25th International Conference on Artificial Intelligence and Statistics},
  pages     = {7659--7679},
  year      = {2022},
  organization = {PMLR}
}

@inproceedings{gidel2019negative,
  title     = {Negative momentum for improved game dynamics},
  author    = {Gidel, Gauthier and Hemmat, Reyhane Askari and Pezeshki, Mohammad and Le Priol, R{\'e}mi and Huang, Gabriel and Lacoste-Julien, Simon and Mitliagkas, Ioannis},
  booktitle = {Proceedings of the 22nd International Conference on Artificial Intelligence and Statistics},
  pages     = {1802--1811},
  year      = {2019},
  publisher = {PMLR}
}

@inproceedings{mokhtari2020unified,
  title     = {A unified analysis of extra-gradient and optimistic gradient methods for saddle point problems: Proximal point approach},
  author    = {Mokhtari, Aryan and Ozdaglar, Asuman and Pattathil, Sarath},
  booktitle = {Proceedings of the 23rd International Conference on Artificial Intelligence and Statistics},
  pages     = {1497--1507},
  year      = {2020},
  publisher = {PMLR}
}

@article{daskalakis2017training,
  title   = {Training GANs with optimism},
  author  = {Daskalakis, Constantinos and Ilyas, Andrew and Syrgkanis, Vasilis and Zeng, Haoyang},
  journal = {arXiv preprint arXiv:1711.00141},
  year    = {2017}
}

@inproceedings{liang2019interaction,
  title     = {Interaction matters: A note on non-asymptotic local convergence of generative adversarial networks},
  author    = {Liang, Tengyuan and Stokes, James},
  booktitle = {Proceedings of the 22nd International Conference on Artificial Intelligence and Statistics},
  pages     = {907--915},
  year      = {2019},
  publisher = {PMLR}
}

@inproceedings{goujaud2023fundamental,
  title     = {On fundamental proof structures in first-order optimization},
  author    = {Goujaud, Baptiste and Dieuleveut, Aymeric and Taylor, Adrien},
  booktitle = {2023 62nd IEEE Conference on Decision and Control (CDC)},
  pages     = {3023--3030},
  year      = {2023},
  organization = {IEEE}
}

@inproceedings{taylor2018lyapunov,
  title={Lyapunov functions for first-order methods: Tight automated convergence guarantees},
  author={Taylor, Adrien and Van Scoy, Bryan and Lessard, Laurent},
  booktitle={International Conference on Machine Learning},
  pages={4897--4906},
  year={2018},
  organization={PMLR}
}

@article{kalman1960control,
  title     = {Control system analysis and design via the “second method” of Lyapunov: I—Continuous-time systems},
  author    = {Kalman, Rudolf E. and Bertram, John E.},
  journal   = {Transactions of the ASME--Journal of Basic Engineering},
  volume    = {82},
  number    = {2},
  pages     = {371--393},
  year      = {1960},
  publisher = {American Society of Mechanical Engineers}
}

@book{boyd2004convex,
  title     = {Convex Optimization},
  author    = {Boyd, Stephen and Vandenberghe, Lieven},
  year      = {2004},
  publisher = {Cambridge University Press}
}

@article{KallioRosa1999,
  title   = {Large-scale convex optimization via saddle point computation},
  author  = {Kallio, M. and Rosa, C. H.},
  journal = {Operations Research},
  volume  = {47},
  number  = {1},
  pages   = {93--101},
  year    = {1999},
  doi     = {10.1287/opre.47.1.93}
}

@article{Nedic2009Subgradient,
  title   = {Subgradient Methods for Saddle-Point Problems},
  author  = {Nedi{\'c}, Angelia and Ozdaglar, Asuman E.},
  journal = {Journal of Optimization Theory and Applications},
  volume  = {142},
  number  = {1},
  pages   = {205--228},
  year    = {2009},
  doi     = {10.1007/s10957-009-9522-7}
}

@article{goodfellow2020generative,
  title   = {Generative adversarial networks},
  author  = {Goodfellow, Ian and Pouget-Abadie, Jean and Mirza, Mehdi and Xu, Bing and Warde-Farley, David and Ozair, Sherjil and Courville, Aaron and Bengio, Yoshua},
  journal = {Communications of the ACM},
  volume  = {63},
  number  = {11},
  pages   = {139--144},
  year    = {2020},
  publisher = {ACM New York, NY, USA}
}

@article{drori2014performance,
  title     = {Performance of first-order methods for smooth convex
               minimization: a novel approach},
  author    = {Drori, Yoel and Teboulle, Marc},
  journal   = {Math. Program.},
  volume    = {145},
  number    = {1--2},
  pages     = {451--482},
  year      = {2014},
  doi       = {10.1007/s10107-013-0653-0}
}

@inproceedings{li2019robust,
  title     = {Robust multi-agent reinforcement learning via minimax deep deterministic policy gradient},
  author    = {Li, Shihui and Wu, Yi and Cui, Xinyue and Dong, Honghua and Fang, Fei and Russell, Stuart},
  booktitle = {Proceedings of the AAAI Conference on Artificial Intelligence},
  pages     = {4213--4220},
  year      = {2019}
}

@article{zhang2021unified,
  title   = {A unified analysis of first-order methods for smooth games via integral quadratic constraints},
  author  = {Zhang, Guodong and Bao, Xuchan and Lessard, Laurent and Grosse, Roger},
  journal = {J. Mach. Learn. Res.},
  volume  = {22},
  number  = {1},
  articleno = {103},
  numpages  = {39},
  year    = {2021},
  publisher = {JMLR.org}
}

@inproceedings{lee2024fundamental,
  title     = {Fundamental benefit of alternating updates in minimax optimization},
  author    = {Lee, Jaewook and Cho, Hanseul and Yun, Chulhee},
  year      = {2024},
  booktitle = {Proceedings of the 41st International Conference on Machine Learning},
  publisher = {PMLR},
  articleno = {1056},
  numpages  = {76},
  address   = {Vienna, Austria},
  series    = {ICML'24}
}

@article{kim2016optimized,
  title   = {Optimized first-order methods for smooth convex minimization},
  author  = {Kim, Donghwan and Fessler, Jeffrey A},
  journal = {Mathematical Programming},
  volume  = {159},
  pages   = {81--107},
  year    = {2016},
  publisher = {Springer}
}

@inproceedings{pmlr-v37-zhanga15,
  title     = {Stochastic Primal-Dual Coordinate Method for Regularized Empirical Risk Minimization},
  author    = {Zhang, Yuchen and Lin, Xiao},
  booktitle = {Proceedings of the 32nd International Conference on Machine Learning},
  pages     = {353--361},
  year      = {2015},
  publisher = {PMLR}
}

@inproceedings{pmlr-v70-scaman17a,
  title     = {Optimal Algorithms for Smooth and Strongly Convex Distributed Optimization in Networks},
  author    = {Scaman, Kevin and Bach, Francis and Bubeck, S{\'e}bastien and Lee, Yin Tat and Massouli{\'e}, Laurent},
  booktitle = {Proceedings of the 34th International Conference on Machine Learning},
  pages     = {3027--3036},
  year      = {2017},
  publisher = {PMLR}
}

@article{peyre2019computational,
  title   = {Computational Optimal Transport: With Applications to Data Science},
  author  = {Peyr{\'e}, Gabriel and Cuturi, Marco},
  journal = {Foundations and Trends in Machine Learning},
  volume  = {11},
  number  = {5--6},
  pages   = {355--607},
  year    = {2019},
  doi     = {10.1561/2200000073},
  url     = {https://doi.org/10.1561/2200000073}
}

@article{chambolle2016introduction,
  title   = {An introduction to continuous optimization for imaging},
  author  = {Chambolle, Antonin and Pock, Thomas},
  journal = {Acta Numerica},
  volume  = {25},
  pages   = {161--319},
  year    = {2016},
  publisher = {Cambridge University Press}
}

\clearpage


\section{Appendix: Generating quadratic Lyapunov functions}

Next, we construct a small-sized SDP to verify linear convergence rate of first-order methods via generation of a quadratic Lyapunov function. The content of this subsection is almost entirely technical. We refer to the original paper \cite{taylor2018lyapunov}, where the framework was introduced for strongly convex setting and discussed in detail.

Lyapunov functions were originally proposed for stability analysis of dynamical systems \cite{kalman1960control}. Now they are an widely used tool for convergence analysis in optimization. 

\begin{definition} \label{fed:lyapunov_func} Consider function $f$ and algorithm that outputs $z_k$ at $k$-th iteration ($k > 0$). Suppose $z_*$ is the optimal point. Lyapunov function is a continuous function $\mathcal V: \sR \to \R$ such that:
\begin{align*}
    &1. \quad \forall z: \ \mathcal V(z) \geqslant 0, \\ 
    &2. \quad \mathcal V(z) = 0 \quad \text{if and only if} \quad z = z_\ast, \\
    &3. \quad \mathcal V(z) \to \infty \quad \text{as} \quad \|z\| \to \infty, \\
    &4. \quad \mathcal V(z_{k+1}) \leqslant \mathcal V(z_{k}) \quad \text{for all} \quad k \geqslant 0. 
\end{align*}
\end{definition}

The notation remains the same except we consider $N$-memory methods \eqref{method:M} that perform $K$ iterations. We denote error vectors: $\textbf{x}_k, \textbf{g}^x_{k}, \textbf{h}^x_{k} \in \R^{(N+1)d_x}$,  $\textbf{y}_k, \textbf{g}^y_{k}, \textbf{h}^y_{k} \in \R^{(N+1)d_y}$, $\textbf{f}^x_k, \textbf{f}^y_k \in \R^{N+1}$:
\begin{align*}
    &\textbf{x}_k := [(x_k - x_*)^\top, ... , (x_{k-N} - x_*)^\top]^\top, \\
    &\textbf{g}^x_{k} := [(g^x_{k} - g^x_{*})^\top, ... , (g^x_{k-N} - g^x_{*})^\top]^\top, \\
    &\textbf{h}^x_{k} := [(h^x_{k} - h^x_{*})^\top, ... , (h^x_{k-N} - h^x_{*})^\top]^\top, \\
    &\textbf{f}^x_k := [f^x_k - f^x_*, ... , f^x_{k-N} - f^x_*]^\top, \\
    &\ldots
\end{align*}
Those vectors constitute \textit{state} of the system at iteration $k$: $\xi_k = (\textbf{x}_k, \textbf{g}^x_{k}, \textbf{f}^x_k, \textbf{h}^x_{k}, \textbf{y}_k, \textbf{g}^y_{k}, \textbf{f}^y_k, \textbf{h}^y_{k})$. 

We consider quadratic Lyapunov functions of the following form ($k \geqslant N$):
\begin{align}
\mathcal{V}(\xi_k) = 
\begin{bmatrix} 
\mathbf{x}_k \\ \mathbf{g}^x_{k} \\ \mathbf{h}^x_{k}
\end{bmatrix}^\top
(P_x \otimes \mathcal I_{d_x})
\begin{bmatrix} 
\mathbf{x}_k \\ \mathbf{g}^x_{k} \\ \mathbf{h}^x_{k}
\end{bmatrix} +
\begin{bmatrix} 
\mathbf{y}_k \\ \mathbf{g}^y_{k} \\ \mathbf{h}^y_{k}
\end{bmatrix}^\top
(P_y \otimes \mathcal I_{d_y})
\begin{bmatrix} 
\mathbf{y}_k \\ \mathbf{g}^y_{k} \\ \mathbf{h}^y_{k}
\end{bmatrix} + p_x^\top \textbf{f}^x_k + p_y^\top \textbf{f}^y_k,
\label{quadratic_lyapunov_family}
\end{align}
where $P_x, P_y \in \SM^{3(N+1)}$ and $p_x, p_y \in \R^{N+1}$. By $\mathcal I_{d_x}$, $\mathcal I_{d_y}$, we denote identity matrices.

The basis vectors 
\begin{align*}
    &\overline{x}^{(K)}_{k}, \overline{g}^{x(K)}_{k}, \overline{h}^{x(K)}_{k} \in \R^{N+2K+3}, \quad \overline{f}^{x(K)}_{k}, \overline{f}^{y(K)}_{k} \in \R^{K+1}, \\
    &\overline{y}^{(K)}_{k}, \overline{g}^{y(K)}_{k}, \overline{h}^{y(K)}_{k} \in \R^{N+2K+3}
\end{align*}
are initialized similarly:
\begin{align*}
    &\overline{x}^{(K)}_{k} = \overline{y}^{(K)}_{k}  := \overline{e}_{k+N+1}^\top \quad k \in {-N,..,0}, \\
    &\overline{g}^{x(K)}_{k} = \overline{g}^{y(K)}_{k}:= \overline{e}_{k+N+2}^\top \quad k \in {0,..,K}, \\
    &\overline{h}^{x(K)}_{k} = \overline{h}^{y(K)}_{k}:= \overline{e}_{k+K+N+3}^\top \quad k \in {0,..,K}, \\
    &\overline{f}^{x(K)}_{k} = \overline{f}^{y(K)}_{k} := \overline{e}_{k+1}^\top \quad k \in {0,..,K}
\end{align*}
for $K \in \{N,N+1\}$. The rest of the vectors are defined in accordance with algorithm \eqref{method:M} that iterates for $k = 0,\ldots,K$ ($K+1$ iteration). To encode interpolation conditions, we introduce $m^{x}_{ij(K)}, m^{y}_{ij(K)} \in \R^{K+1}$ and $M^{x}_{ij(K)}, M^{y}_{ij(K)} \in \SM^{N+2K+3}$: 
\begin{align*}
    &m^{x(K)}_{ij} := (L_x-\mu_x)(\overline{f}^{x(K)}_{i} -\overline{f}^{x(K)}_{j})^\top, \quad m^{y(K)}_{ij} := (L_y-\mu_y)(\overline{f}^{y(K)}_{i} -\overline{f}^{y(K)}_{j})^\top, \\
    &M^{x(K)}_{ij} := 
    \begin{bmatrix} 
        \overline{u}^{x(K)}_i\\ \overline{u}^{x(K)}_j\\ \overline{g}^{x(K)}_i\\ 
        \overline{g}^{x(K)}_j
    \end{bmatrix}^\top
    M^x
    \begin{bmatrix} 
        \overline{u}^{x(K)}_i\\ \overline{u}^{x(K)}_j\\ \overline{g}^{x(K)}_i\\ 
        \overline{g}^{x(K)}_j
    \end{bmatrix} ,\quad
    M^{y(K)}_{ij} := 
    \begin{bmatrix} 
        \overline{u}^{y(K)}_i\\ \overline{u}^{y(K)}_j\\ \overline{g}^{y(K)}_i\\ 
        \overline{g}^{y(K)}_j
    \end{bmatrix}^\top
    M^y
    \begin{bmatrix} 
        \overline{u}^{y(K)}_i\\ \overline{u}^{y(K)}_j\\ \overline{g}^{y(K)}_i\\ 
        \overline{g}^{y(K)}_j
    \end{bmatrix},
\end{align*}
where $M^x, M^y \in \SM^4$ are defined as \eqref{eq:interp_matrix}. To encode interpolation constraints for bilinear part, we define $B^{(K)}_{Hx}, B^{(K)}_{Hy} \in \R^{(N+2K+3)\times(K+1)}$, $B^{(K)}_{vx}, B^{(K)}_{vy} \in \R^{(N+2K+3)\times(K+1)}$:
\begin{align*}
    &B^{(K)}_{Hx} = \left[(\overline{h}^{x(K)}_0)^\top, \quad ..., \quad (\overline{h}^{x(K)}_K)^\top\right], \quad B^{(K)}_{Hy} = \left[(\overline{h}^{y(K)}_0)^\top, \quad ..., \quad (\overline{h}^{y(K)}_K)^\top\right], \\
    &B^{(K)}_{vx} = \left[(\overline{v}^{x(K)}_{0}){}^\top, \quad ..., \quad (\overline{v}^{x(K)}_{K}){}^\top\right], \quad
    B^{(K)}_{vy} = \left[(\overline{v}^{y(K)}_{0}){}^\top, \quad ..., \quad (\overline{v}^{y(K)}_{K}){}^\top\right].
\end{align*}

Next, write a Lyapunov function and its decrease after one iteration. We introduce vectors $v^{x(K)}_k, v^{y(K)}_k \in \R^{K+1}$ and matrices $\mathcal V^{x(K)}_k, \mathcal V^{y(K)}_k \in \SM^{N+2K+3}$ that correspond to linear and quadratic parts respectively:
\begin{align*}
    &v^{x(K)}_k := p_x^\top \overline{\mathbf{f}}^{x(K)}_k, \quad v^{y(K)}_k := p_y^\top\overline{\mathbf{f}}^{y(K)}_k, \\ 
    &\mathcal V^{x(K)}_k := 
    \begin{bmatrix} 
        \overline{\mathbf{x}}^{(K)}_k \\ \overline{\mathbf{g}}^{x(K)}_k \\
        \overline{\mathbf{h}}^{x(K)}_k 
    \end{bmatrix}^\top
    P_x
    \begin{bmatrix} 
        \overline{\mathbf{x}}^{(K)}_k \\ \overline{\mathbf{g}}^{x(K)}_k \\
        \overline{\mathbf{h}}^{x(K)}_k 
    \end{bmatrix}, \quad 
    \mathcal V^{y(K)}_k := 
    \begin{bmatrix} 
        \overline{\mathbf{y}}^{(K)}_k \\ \overline{\mathbf{g}}^{y(K)}_k \\
        \overline{\mathbf{h}}^{y(K)}_k 
    \end{bmatrix}^\top
    P_y
    \begin{bmatrix} 
        \overline{\mathbf{y}}^{(K)}_k \\ \overline{\mathbf{g}}^{y(K)}_k \\
        \overline{\mathbf{h}}^{y(K)}_k 
    \end{bmatrix}.
\end{align*}
Matrices $\overline{\mathbf{x}}^{(K)}_k, \overline{\mathbf{g}}^{x(K)}_k, \overline{\mathbf{h}}^{x(K)}_k \in \R^{(N+1)\times(N+2K+3)}$, $\overline{\mathbf{y}}^{(K)}_k, \overline{\mathbf{g}}^{y(K)}_k,
\overline{\mathbf{h}}^{y(K)}_k \in \R^{(N+1)\times(N+2K+3)}$ and $\overline{\mathbf{f}}^{x(K)}_k, \overline{\mathbf{f}}^{y(K)}_k \in \R^{(N+1)\times(K+1)}$ are defined as follows:

\begin{align*}
    &\overline{\mathbf{x}}^{(K)}_k =
\begin{bmatrix} 
\overline{x}^{(K)}_k\\.\\.\\ \overline{x}^{(K)}_{k-N} 
\end{bmatrix}, \quad 
\overline{\mathbf{g}}^{x(K)}_k =
\begin{bmatrix} 
\overline{g}^{x(K)}_k\\.\\.\\ \overline{g}^{x(K)}_{k-N} 
\end{bmatrix}, \quad 
\overline{\mathbf{h}}^{x(K)}_k =
\begin{bmatrix} 
\overline{h}^{x(K)}_k\\.\\.\\ \overline{h}^{x(K)}_{k-N} 
\end{bmatrix}, \quad 
\overline{\mathbf{f}}^{x(K)}_k =
\begin{bmatrix} 
\overline{f}^{x(K)}_k \\.\\.\\ \overline{f}^{x(K)}_{k-N} 
\end{bmatrix}, \\
&\ldots
\end{align*}
The decrease of quadratic and linear parts of Lyapunov function:
\begin{align*}
    &\delta v^{x(K)}_k := v^{x(K)}_{k+1} - \rho^2 v^{x(K)}_{k}, \quad 
    \delta v^{y(K)}_k := v^{y(K)}_{k+1} - \rho^2 v^{y(K)}_{k}, \\
    &\delta \mathcal V^{x(K)}_k := \mathcal V^{x(K)}_{k+1} - \rho^2 \mathcal V^{x(K)}_{k}, \quad \delta \mathcal V^{y(K)}_k := \mathcal V^{y(K)}_{k+1} - \rho^2 \mathcal V^{y(K)}_{k},
\end{align*}
where $\rho$ is the rate of linear convergence, that we seek to certify. Finally, we define index sets $I_K := \{0,1,\ldots,K,\ast \}$ and variables:
\begin{align*}
    &P_x,P_y \in \SM^{3(N+1)}, \quad p_x, p_y \in \R^{N+1}, \tag{Lyap-vars}\\
    &C_1, C_2, C_3 \in \SM^{N+1}, \quad \lambda_{ij}, \overline{\lambda}_{ij} \in \R \quad \text{for } \hspace{2pt} i,j \in I_N, \\
    &F_1, F_2, F_3 \in \SM^{N+2}, \quad \eta_{ij}, \overline{\eta}_{ij} \in \R \quad \text{for } \hspace{2pt} i,j \in I_{N+1}.\\
\end{align*}
Contraction factor $\rho$, step sizes of \eqref{method:M}, and functional constants are parameters of the problem. We present the program itself:

\begin{align*} \label{problem:lyapunov}\tag{Lyap-SDP}
    &\textbf{feasible} \hspace{4pt} (\text{Lyap-vars})  \\
    &0 \leqslant \lambda_{ij}, \quad 0 \leqslant \overline{\lambda}_{ij} \quad \forall i,j \in I_N, \\
    &0 \leqslant \eta_{ij}, \quad 0 \leqslant \overline{\eta}_{ij} \quad \forall i,j \in I_{N+1}, \\
    &0 \preceq C_2, \quad 0 \preceq C_3, \quad 0 \preceq F_2, \quad 0 \preceq F_3, \\
    &0 \prec \mathcal V^{x(N)}_N - \sum_{i,j\in I_N} \lambda_{ij} M^{x(N)}_{ij} + B^{(N)}_{vx} C_1 B^{(N)}_{Hx}{}^\top + B^{(N)}_{Hx} C_2 B^{(N)}_{Hx}{}^\top - L_{xy}^2 B^{(N)}_{vx} C_3 B^{(N)}_{vx}{}^\top, \\
    &0 \prec \mathcal V^{y(N)}_N - \sum_{i,j \in I_N} \overline{\lambda}_{ij} M^{y(N)}_{ij} - B^{(N)}_{Hy} C_1 B^{(N)}_{vy}{}^\top - L_{xy}^2 B^{(N)}_{vy}C_2 B^{(N)}_{vy}{}^\top + B^{(N)}_{Hy} C_3 B^{(N)}_{Hy}{}^\top, \\
    &0 < v^{x(N)}_N - \sum_{i,j \in I_N} \lambda_{ij} m^{x(N)}_{ij}, \\
    &0  < v^{y(N)}_N - \sum_{i,j \in I_N} \overline{\lambda}_{ij} m^{y(N)}_{ij}, \\
    &0 \succeq \delta \mathcal V^{x(N+1)}_N + \sum_{i,j\in I_{N+1}} \eta_{ij} M^{x(N+1)}_{ij} \\&\quad - B^{(N+1)}_{vx} F_1 B^{(N+1)}_{Hx}{}^\top - B^{(N+1)}_{Hx} F_2 B^{(N+1)}_{Hx}{}^\top + L_{xy}^2 B^{(N+1)}_{vx} F_3 B^{(N+1)}_{vx}{}^\top, \\
    &0 \succeq \delta \mathcal V^{y(N+1)}_N + \sum_{i,j\in I_{N+1}} \overline{\eta}_{ij} M^{y(N+1)}_{ij} \\&\quad + B^{(N+1)}_{Hy} F_1 B^{(N+1)}_{vy}{}^\top + L_{xy}^2 B^{(N+1)}_{vy} F_2 B^{(N+1)}_{vy}{}^\top - B^{(N+1)}_{Hy} F_3 B^{(N+1)}_{Hy}{}^\top, \\
    &0 \geqslant \delta v^{x(N+1)}_N + \sum_{i,j \in I_{N+1}} \eta_{ij} m^{x(N+1)}_{ij}, \\
    &0 \geqslant \delta v^{y(N+1)}_N + \sum_{i,j \in I_{N+1}} \overline{\eta}_{ij} m^{y(N+1)}_{ij}.
\end{align*}

The following theorem states that feasibility of \eqref{problem:lyapunov} is equivalent to existence of a quadratic Lyapunov function of the
form~\eqref{quadratic_lyapunov_family} that decreases (at least) by the factor of $\rho^2$ at each iteration.

\begin{theorem} [Correctness of \eqref{problem:lyapunov}] Consider applying algorithm \eqref{method:M} to $f \in \mathcal K_{\mu_x,\mu_y,L_x,L_y,L_{xy}}$. If \eqref{problem:lyapunov} is feasible then  there exists a quadratic Lyapunov function of the
form~\eqref{quadratic_lyapunov_family} with contraction $\rho$ that is valid for all $d_x, d_y \in \N$.
\label{teo_lyapunov_main}
\end{theorem}

With Theorem~\ref{teo_lyapunov_main}, we can perform bisection on $\rho$ to find
the minimum $\rho$ such that \eqref{problem:lyapunov} is feasible to obtain
the fastest linear convergence rate.

\subsection{Proof of Theorem~\ref{teo_lyapunov_main}}

Define vectors $\textbf{x} \in \R^{(N+1)d_x}$, $\textbf{g}^x, \textbf{h}^x \in \R^{(K+1)d_x}$, $\textbf{y} \in \R^{(N+1)d_y}$, $\textbf{g}^y, \textbf{h}^y \in \R^{(K+1)d_y}$, $\textbf{f}^x, \textbf{f}^y\in \R^{K+1}$:
\begin{align*}
    &\textbf{x} := \left[(x_{-N} - x_*)^\top, ... , (x_{0} - x_*)^\top\right]^\top, \\
    &\textbf{g}^x := \left[(g^x_{0} - g^x_{*})^\top, ... , (g^x_{K} - g^x_{*})^\top\right]^\top, \\
    &\textbf{h}^x := \left[(h^x_{0} - h^x_{*})^\top, ... , (h^x_{K} - h^x_{*})^\top\right]^\top, \\
    &\textbf{f}^x := \left[f^x_0 - f^x_*, ... , f^x_{K} - f^x_*\right]^\top, \\
    &\ldots
\end{align*}
Denote state $\xi := (\textbf{x},\ \textbf{g}^x, \textbf{f}^x, \textbf{y}, \textbf{g}^y, \textbf{f}^y,    \textbf{h}^x, \textbf{h}^y)$. For convenience, we also denote $\overline{\xi}_x = \left[\textbf{x}^\top,\ \textbf{g}^x{}^\top, \textbf{h}^x{}^\top\right]^\top$ and $\overline{\xi}_y = \left[\textbf{y}^\top,\ \textbf{g}^y{}^\top, \textbf{h}^y{}^\top\right]^\top$. The standard Gram matrix, containing all inner products between $x_i-x_*$ for $i \in \{-N,\ldots, 0\}$ and $g^x_i$, $h^x_i$ for $i \in \{0,\ldots, K\}$: $G_x = W_x^\top W_x \in \SM^{N+2K+3}$, where
\begin{align*}
    W_x := \left[x_{-N} - x_* \quad ... \quad x_0 - x_* \quad g^x_{0} \quad ... \quad g^x_{K} \quad h^x_{0} \quad ... \quad h^x_{K}\right].
\end{align*}
Gram matrix $G_y = W_y^\top W_y \in \SM^{N+2K+3}$ is defined the same way. \par 

By Theorem~\ref{teo:bil_coupling_int} and using the newly defined notation, $f \in \mathcal K_{\mu_x,\mu_y,L_x,L_y,L_{xy}}$ is equivalent to:
\begin{align}
    &0 \leqslant \phi^{x}_{ij} =  m^{x}_{ij}{}^\top \textbf{f}^x + \text{tr}\left(G_x M^{x}_{ij}\right) \quad \forall i,j \in I_K, \label{app:eq:inter_cond} \\
    &0 \leqslant \phi^{y}_{ij} = m^{y}_{ij}{}^\top \textbf{f}^y + \text{tr}\left(G_y M^{y}_{ij}\right) \quad \forall i,j \in I_K,\notag \\
    &B_{Hx}^\top G_x B_{vx} = B_{vy}^\mathrm{T}G_yB_{Hy}, \notag \\
    &B_{Hx}^\top G_x B_{Hx} \preceq L_{xy}^2 B_{vy}^\top G_yB_{vy}, \notag\\ 
    &B_{Hy}^\top G_y B_{Hy} \preceq L_{xy}^2 B_{vx}^\top G_xB_{vx}, \notag
\end{align}
where $I_K = \{0,\ldots,K,*\}$. \par

\begin{theorem} \label{app:teo:pos_quadratics} Suppose $f \in \mathcal K_{\mu_x,\mu_y,L_x,L_y,L_{xy}}$. Consider
\begin{align}
    \sigma(\xi) = \overline{\xi}_x^
    \top \left(Q_x \otimes \mathcal I_{d_x}\right) \overline{\xi}_x + \overline{\xi}_y^
    \top \left(Q_y \otimes \mathcal I_{d_y}\right) \overline{\xi}_y + q_x^\top \textbf{f}^x + q_y^\top \textbf{f}^y,
\label{app:eq:lyap_quadratic}
\end{align}
where $Q_x, Q_y \in \SM^{N+2K+3}$ and $q_x, q_y \in \R^{K+1}$. If there exist $\lambda_{ij}, \overline{\lambda}_{ij} \in \R$, and $C_1, C_2, C_3 \in \SM^{K+1}$ such that
\begin{align}
    & \lambda_{ij} \geqslant 0,\ \overline{\lambda}_{ij} \geqslant 0 \quad \forall i,j \in I_K, \\
    &C_2 \succeq 0,\  C_3 \succeq 0, \\
    &0 \preceq Q_x - \sum_{i,j\in I_K} \lambda_{ij} M^{x}_{ij} + B_{vx} C_1 B_{Hx}^\top + B_{Hx} C_2 B_{Hx}^\top - L_{xy}^2 B_{vx} C_3 B_{vx}^\top, \label{app:eq:pos_quad_1}\\
    &0 \preceq Q_y - \sum_{i,j \in I_K} \overline{\lambda}_{ij} M^{y}_{ij} - B_{Hy} C_1 B_{vy}^\top - L_{xy}^2 B_{vy}C_2 B_{vy}^\top + B_{Hy} C_3 B_{Hy}^\top, \label{app:eq:pos_quad_2}\\
    &0 \leqslant q_x - \sum_{i,j \in I_K} \lambda_{ij} m^{x}_{ij}, \label{app:eq:pos_quad_3}\\
    &0 \leqslant q_y - \sum_{i,j \in I_K} \overline{\lambda}_{ij} m^{y}_{ij}, \label{app:eq:pos_quad_4}
\end{align}
then $\sigma(\xi) \geqslant 0$.
\end{theorem}
\begin{proof} 
First, apply $\overline{\xi}_x^\top \left(\cdot \otimes \mathcal I_{d_x}\right) \overline{\xi}_x$ and $\overline{\xi}_y^\top \left(\cdot \otimes \mathcal I_{d_y}\right) \overline{\xi}_y$ to inequalities \eqref{app:eq:pos_quad_1} and \eqref{app:eq:pos_quad_2}. Also, multiply inequalities \eqref{app:eq:pos_quad_3}, \eqref{app:eq:pos_quad_4} by $\textbf{f}^x{}^\top$ and $\textbf{f}^y{}^\top$ respectively. Summation of the resulting expressions yields the desired statement.
\end{proof}

Now we show that the reverse statement holds if we assume $\min(d_x,d_y) \geqslant N+2K+3$ and absence of duality gap. Suppose $\sigma(\xi) > 0$, which is equivalent to $\sigma_{\ast}(\varepsilon) > 0$ for all $\varepsilon > 0$, where
\begin{align*} 
\sigma_{\ast}^{(d_x,d_y)}(\varepsilon) &:= \underset{\xi}{\min} \ \sigma(\xi) \quad \text{s.t.} \\
&f \in \mathcal K_{\mu_x,\mu_y,L_x,L_y,L_{xy}}, \\
&\xi \quad \text{is generated by \eqref{method:M}}, \\
&\|\textbf{x}\|^2_2 + \|\textbf{g}^x\|^2_2 + \|\textbf{h}^x\|^2_2 + \|\textbf{y}\|^2_2 + \|\textbf{g}^y\|^2_2 + \|\textbf{h}^y\|^2_2 + \textbf{1}^\top \textbf{f}^x + \textbf{1}^\top \textbf{f}^y = \varepsilon.
\end{align*}
Express $\sigma(\xi)$ is terms of Gram matrices:
\begin{align*}
    \sigma(G_x,G_y,\textbf{f}^x,\textbf{f}^y) =  \text{tr}\left(G_x Q_x\right) + \text{tr}\left(G_y Q_y\right) + q_x^\top \textbf{f}^x + q_y^\top \textbf{f}^y,
\end{align*}
and substitute the interpolation conditions \eqref{app:eq:inter_cond}. Also, as was discussed in \cite{taylor2017smooth, taylor2018lyapunov}, as well as in Section~\ref{sec:pep_construction}, we want the problem's value to not depend on dimensions. Because we assumed $\min(d_x,d_y) \geq N+2K+3$, we can safely drop rank constraints:
\begin{align*} 
\sigma_{\ast}^{(\infty,\infty)}(\varepsilon) &:= \underset{G_x,G_y,\textbf{f}^x,\textbf{f}^y}{\min} \ \text{tr}\left(G_x Q_x\right) + \text{tr}\left(G_y Q_y\right) + q_x^\top \textbf{f}^x + q_y^\top \textbf{f}^y \quad \text{s.t.} \\
& 0 \leqslant m^{x}_{ij}{}^\top \textbf{f}^x + \text{tr}\left(G_x M^{x}_{ij}\right) \quad \forall i,j \in I, \\
&0 \leqslant m^{y}_{ij}{}^\top \textbf{f}^y + \text{tr}\left(G_y M^{y}_{ij}\right) \quad \forall i,j \in I, \\
&B_{Hx}^\top G_x B_{vx} = B_{vy}^\mathrm{T}G_yB_{Hy}, \\
&B_{Hx}^\top G_x B_{Hx} \preceq L_{xy}^2 B_{vy}^\top G_yB_{vy}, \\
&B_{Hy}^\top G_y B_{Hy} \preceq L_{xy}^2 B_{vx}^\top G_xB_{vx}, \\
&\text{tr}(G_x) + \text{tr}(G_y) + \textbf{1}^\top \textbf{f}^x + \textbf{1}^\top \textbf{f}^y = \varepsilon, \\
&\textbf{f}^x \geqslant 0, \ \textbf{f}^y \geqslant 0, \\
&G_x \succeq 0,\  G_y \succeq 0.
\end{align*}

Suppose Slater's condition holds (there exists feasible point such that $G_x \succ 0$ and $G_y \succ 0$) \cite{boyd2004convex}. Then the optimal value of the dual problem is equal to the value of the primal problem. The dual value is given by

\begin{align*} 
d^{(\infty,\infty)}(\varepsilon) &:= \underset{\{\lambda_{ij}\}, \{\overline{\lambda}_{ij}\}, \nu, C_1, C_2, C_3}{\max}  \nu \ \varepsilon \quad \text{s.t.} \\
&\lambda_{ij} \geqslant 0,\  \overline{\lambda}_{ij} \geqslant 0 \quad \forall i,j \in I, \\
&C_2 \succeq 0,\ C_3 \succeq 0, \\
&Q_x - \sum_{i,j \in I} \lambda_{ij} M^{x}_{ij} + B_{vx} C_1 B_{Hx}^\top + B_{Hx} C_2 B_{Hx}^\top - L_{xy}^2 B_{vx} C_3 B_{vx}^\top \succeq \nu I_{N+2K+3}, \\
&Q_y - \sum_{i,j \in I} \overline{\lambda}_{ij} M^{y}_{ij} - B_{Hy} C_1 B_{vy}^\top - L_{xy}^2 B_{vy}C_2 B_{vy}^\top + B_{Hy} C_3 B_{Hy}^\top \succeq \nu I_{N+2K+3}, \\
&q_x - \sum_{i,j \in I} \lambda_{ij} m^{x}_{ij} \geqslant \nu \textbf{1}_{K+1}, \\
&q_y - \sum_{i,j \in I} \overline{\lambda}_{ij} m^{y}_{ij} \geqslant \nu \textbf{1}_{K+1}.
\end{align*}
From $\sigma_{\ast}(\varepsilon) > 0$ for all $\varepsilon > 0$ and strong duality, we have
\begin{align*}
    &Q_x - \sum_{i,j \in I} \lambda_{ij} M^{x}_{ij} + B_{vx} C_1 B_{Hx}^\top + B_{Hx} C_2 B_{Hx}^\top - L_{xy}^2 B_{vx} C_3 B_{vx}^\top \succ 0, \\
    &Q_y - \sum_{i,j \in I} \overline{\lambda}_{ij} M^{y}_{ij} - B_{Hy} C_1 B_{vy}^\top - L_{xy}^2 B_{vy}C_2 B_{vy}^\top + B_{Hy} C_3 B_{Hy}^\top \succ 0, \\
    &q_x - \sum_{i,j \in I} \lambda_{ij} m^{x}_{ij} > 0, \\
    &q_y - \sum_{i,j \in I} \overline{\lambda}_{ij} m^{y}_{ij} > 0.
\end{align*}

This concludes the proof of necessity. Moving forward with the proof of Theorem~\ref{teo_lyapunov_main}, the expression~\eqref{quadratic_lyapunov_family} for Lyapunov function's value at $\xi_k$:
\begin{align*}
    \mathcal{V}(\xi_k) = \overline{\xi}_x^
    \top \left(\mathcal{V}^{x(K)}_k \otimes \mathcal I_{d_x}\right) \overline{\xi}_x + \overline{\xi}_y^
    \top \left(\mathcal{V}^{y(K)}_k \otimes \mathcal I_{d_y}\right) \overline{\xi}_y + (v^{x(K)}_k)^\top \textbf{f}^x + (v^{y(K)}_k)^\top \textbf{f}^y.
\end{align*}
The expression for the (additional) decrease $\delta \mathcal{V}(\xi_k) := \mathcal{V}(\xi_{k+1}) - \rho^2 \mathcal{V}(\xi_{k})$ after one iteration:
\begin{align*}
    \delta \mathcal{V}(\xi_k) = \overline{\xi}_x^
    \top \left(\delta \mathcal{V}^{x(K)}_k \otimes \mathcal I_{d_x}\right) \overline{\xi}_x + \overline{\xi}_y^
    \top \left(\delta \mathcal{V}^{y(K)}_k \otimes \mathcal I_{d_y}\right) \overline{\xi}_y + (\delta v^{x(K)}_k)^\top \textbf{f}^x + (\delta v^{y(K)}_k)^\top \textbf{f}^y.
\end{align*}
From Theorem~\ref{app:teo:pos_quadratics}, $\mathcal{V}(\xi_N) > 0$ holds if there exist such $\lambda_{ij}, \overline{\lambda}_{ij} \in \R$, and $C_1, C_2, C_3 \in \SM^{N+1}$ that
\begin{align*}
    &\lambda_{ij} \geqslant 0,\ \overline{\lambda}_{ij} \geqslant 0 \quad \forall i,j \in I_N, \\
    &C_2 \succeq 0,\  C_3 \succeq 0, \\
    &0 \prec \mathcal V^{x(N)}_N - \sum_{i,j\in I_N} \lambda_{ij} M^{x(N)}_{ij} + B^{(N)}_{vx} C_1 B^{(N)}_{Hx}{}^\top + B^{(N)}_{Hx} C_2 B^{(N)}_{Hx}{}^\top - L_{xy}^2 B^{(N)}_{vx} C_3 B^{(N)}_{vx}{}^\top, \\
    &0 \prec \mathcal V^{y(N)}_N - \sum_{i,j \in I_N} \overline{\lambda}_{ij} M^{y(N)}_{ij} - B^{(N)}_{Hy} C_1 B^{(N)}_{vy}{}^\top - L_{xy}^2 B^{(N)}_{vy}C_2 B^{(N)}_{vy}{}^\top + B^{(N)}_{Hy} C_3 B^{(N)}_{Hy}{}^\top, \\
    &0 < v^{x(N)}_N - \sum_{i,j \in I_N} \lambda_{ij} m^{x(N)}_{ij}, \\
    &0  < v^{y(N)}_N - \sum_{i,j \in I_N} \overline{\lambda}_{ij} m^{y(N)}_{ij},
\end{align*}
Next, $\delta \mathcal{V}(\xi_N) \leqslant 0$ holds if there exist such $\eta_{ij}, \overline{\eta}_{ij} \in \R$, and $F_1, F_2, F_3 \in \SM^{N+2}$ that
\begin{align*}
    &\eta_{ij} \geqslant 0,\ \overline{\eta}_{ij} \geqslant 0 \quad \forall i,j \in I_{N+1}, \\
    &F_2 \succeq 0,\  F_3 \succeq 0, \\
    &0 \succeq \delta \mathcal V^{x(N+1)}_N + \sum_{i,j\in I_{N+1}} \eta_{ij} M^{x(N+1)}_{ij} \\&\quad - B^{(N+1)}_{vx} F_1 B^{(N+1)}_{Hx}{}^\top - B^{(N+1)}_{Hx} F_2 B^{(N+1)}_{Hx}{}^\top + L_{xy}^2 B^{(N+1)}_{vx} F_3 B^{(N+1)}_{vx}{}^\top, \\
    &0 \succeq \delta \mathcal V^{y(N+1)}_N + \sum_{i,j\in I_{N+1}} \overline{\eta}_{ij} M^{y(N+1)}_{ij} \\&\quad + B^{(N+1)}_{Hy} F_1 B^{(N+1)}_{vy}{}^\top + L_{xy}^2 B^{(N+1)}_{vy} F_2 B^{(N+1)}_{vy}{}^\top - B^{(N+1)}_{Hy} F_3 B^{(N+1)}_{Hy}{}^\top, \\
    &0 \geqslant \delta v^{x(N+1)}_N + \sum_{i,j \in I_{N+1}} \eta_{ij} m^{x(N+1)}_{ij}, \\
    &0 \geqslant \delta v^{y(N+1)}_N + \sum_{i,j \in I_{N+1}} \overline{\eta}_{ij} m^{y(N+1)}_{ij}.
\end{align*}

If we combine all the constraints above, we get the constraints of the problem~\eqref{problem:lyapunov}. Therefore, if the program is feasible then there exists such quadratic Lyapunov function~\eqref{quadratic_lyapunov_family} that $\mathcal{V}(\xi_N) > 0$ and $\delta \mathcal{V}(\xi_N) \leqslant 0$. This concludes the proof of Theorem~\ref{teo_lyapunov_main}. As we discussed above, under additional assumptions of $\min(d_x,d_y) \geqslant N+2K+3$ and strong duality: if the program is infeasible then such quadratic Lyapunov function does not exist.

\end{document}